\documentclass[hidelinks,11pt]{article}

\usepackage{lipsum}
\usepackage{amsfonts}
\usepackage{graphicx}
\usepackage{epstopdf}
\usepackage{algorithmic}
\usepackage{booktabs}
\usepackage{subcaption}
\usepackage{esvect}
\usepackage{multirow}
\usepackage{amsmath}
\usepackage{amsmath}
\usepackage{amssymb}  
\usepackage{caption}

\newcommand{\bol}{\boldsymbol}

\newcommand{\ney}{\boldsymbol{y}}                          
\newcommand{\nex}{\boldsymbol{x}}           
\newcommand{\bnex}{\bold{x}} 
\newcommand{\btau}{\boldsymbol{\tau}}

\newcommand{\ner}{\boldsymbol{r}}

\newcommand{\de}{\,\mathrm{d}}                               
\newcommand{\e}{\operatorname{e}}                               
\newcommand{\im}{\operatorname{i}}                     
\newcommand{\inc}{\mathrm{inc}}

\newcommand{\andtext}{\quad\mbox{and}\quad}

\newcommand{\p}{\partial}

\newcommand{\lf}{\left}
\newcommand{\rg}{\right}

\newcommand{\R}{\mathbb{R}}       
\newcommand{\C}{\mathbb{C}}

\newcommand{\bxi}{\boldsymbol\xi}  
\newcommand{\bnor}{\bold n}          
\newcommand{\nor}{\boldsymbol n} 
\newcommand{\curl}{\operatorname{curl}}

\usepackage{algorithm}
\usepackage{multirow}
\usepackage{amsmath}
\usepackage{amsfonts}
\usepackage{amsmath}
\usepackage{amssymb}  
\usepackage{graphicx}
\usepackage{caption}
\usepackage{enumitem}
\usepackage{mathrsfs}
\usepackage{upgreek}
\usepackage{amsthm}
\usepackage{booktabs}
\usepackage{authblk}
\usepackage[noabbrev]{cleveref}
\crefformat{equation}{(#2#1#3)}

\newtheorem{theorem}{Theorem}[section]
\newtheorem{lemma}[theorem]{Lemma}

\newtheorem{remark}[theorem]{Remark}

\usepackage{xcolor}
\usepackage[normalem]{ulem}

\title{Planewave density interpolation methods for 3D Helmholtz boundary integral equations}
\author[1]{Carlos P\'erez-Arancibia\thanks{cperez@mat.uc.cl}}
\affil[1]{\small{Institute for Mathematical and Computational Engineering, School of Engineering and Faculty of Mathematics, Pontificia Universidad Cat\'olica de Chile}}
\author[2]{Catalin Turc\thanks{catalin.c.turc@njit.edu}}
\affil[2]{\small{Department of Mathematical Sciences, New Jersey Institute of Technology}}
\author[3]{Luiz M. Faria\thanks{luiz.maltez-faria@inria.fr}}
\affil[3]{\small{Laboratoire POEMS, INRIA}}

\date{\today}

\begin{document}

\maketitle

\begin{abstract}
This paper introduces planewave density interpolation methods for the regularization of weakly singular, strongly singular, hypersingular and nearly singular integral kernels present in 3D Helmholtz surface layer potentials and associated integral operators. Relying on Green's third identity and pointwise interpolation of density functions in the form of planewaves, these methods allow layer potentials and integral operators to be expressed in terms of integrand functions that remain smooth (at least bounded) regardless the location of the target point relative to the surface sources. Common challenging integrals that arise in both Nystr\"om and boundary element discretization of boundary integral equation, can then be numerically evaluated by standard quadrature rules that are irrespective of the kernel singularity. Closed-form and purely numerical planewave density interpolation procedures are presented in this paper, which are used in conjunction with Chebyshev-based Nystr\"om and Galerkin boundary element methods. A variety of numerical examples---including problems of acoustic scattering involving multiple touching and even intersecting obstacles, demonstrate the capabilities of the proposed technique.
\end{abstract}

   

\section{Introduction}
Challenging weakly singular, strongly singular, hypersingular, and nearly singular surface integrals are ubiquitous to boundary integral equation (BIE) formulations of linear partial differential equations (PDEs).  A plethora of numerical and semi-analytical procedures including, for instance, singularity subtraction\footnote{This technique has also been referred to as singularity extraction by some authors.}~\cite{caorsi1993theoretical,graglia1993numerical,jarvenpaa2006singularity,Jarvenpaa:2003cs,wilton1984potential,YlaOijala:2003bn}, Duffy-like transformations~\cite{duffy1982quadrature,Sauter:2001iw,Sauter:1996iw,reid2015generalized,hackbusch1993efficient}, polar singularity cancelation~\cite{Bruno:2001ima,Hackbusch:1994tq,schwab1992numerical}, singularity extraction~\cite{Schulz:1998kl,Schwab:bj}, among others techniques, have been proposed in the literature for the evaluation of these difficult integrals in the context of both Nystr\"om and boundary element methods. 
Despite all these significant efforts and the compelling advantages that BIE methods offer over standard volume discretization techniques such as finite element and finite difference methods---especially in handling unbounded domains and seamlessly incorporating radiation conditions at infinity for time-harmonic wave scattering---they still face criticism of being difficult to implement. From the authors' viewpoint, the main source of practical difficulties arises from the significant effort researchers and practitioners have to invest into understanding and implementing a specific set of techniques tailored to handle the various integration scenarios concerning the target point location relative to the surface sources. We hereby address this issue for both  Nystr\"om and boundary element discretizations of three-dimensional Helmholtz BIEs by introducing a universal semi-analytical procedure capable of regularizing all the aforementioned challenging surface integrals at the continuous level, i.e., prior to numerical integration. For the sake conciseness, we specifically consider combined-field BIE formulations of sound-soft (Dirichlet) and sound-hard (Neumann) scattering problems leading to the well-known Brakhage-Werner~\cite{brakhage1965dirichletsche} and Burton-Miller~\cite{Burton1971Application} integral equations, respectively, which feature all four boundary integral operators of Calder\'on calculus. 

As mentioned above, there is extensive literature on the subject. We refer the reader to~\cite{HDI3D} for a thorough review concerning Nystr\"om methods and various approaches to deal with nearly singular integrals. Regarding BEMs specifically, two main groups of techniques can be distinguished. On one hand we have semi-analytical techniques~\cite{caorsi1993theoretical,graglia1993numerical,jarvenpaa2006singularity,Jarvenpaa:2003cs,wilton1984potential,YlaOijala:2003bn} whereby singular terms are extracted from the kernel to be integrated in closed form, while the remaining smoother part is integrated numerically by means of standard quadrature rules. And, on the other hand, we have techniques based on regularizing coordinate transformations~\cite{duffy1982quadrature,Sauter:2001iw,Sauter:1996iw,reid2015generalized,hackbusch1993efficient,Hackbusch:1994tq,schwab1992numerical} whereby specialized changes of variables are utilized to turn singular integrands into regular (analytic) integrands to which standard quadrature rules can be directly applied to achieve any desired accuracy. Although effective at dealing with the specific classes of integrands (on polygonal surface meshes) and basis functions for which they have been designed, none of the aforementioned techniques handle nearly singular integrals arising when target points lying off the surface are close to the surface sources.

In detail, this paper presents planewave density interpolation (PWDI) methods for the regularization of weakly singular, strongly singular, hypersingular  and nearly singular integral kernels present in Helmholtz layer potentials and associated boundary integral operators. Relying on Green's third identity and a certain Taylor-like interpolation of the surface density in terms of homogeneous solutions of the underlaying PDE (planewaves in this case), density interpolation methods~\cite{plane-wave:2018,HDI3D} allow layer potentials and operators to be expressed in terms of integrand functions that are smooth (at least bounded) regardless the target point location. The resulting surface integrals can then be numerically evaluated by means of standard off-the-shelf quadrature rules that are irrespective of the singularity of the associated integral kernels. As such, kernel-regularized layer potentials and operators can be directly evaluated at target points that are arbitrarily close to their surface sources enabling, in particular,  the straightforward Nystr\"om or boundary element method (BEM) discretization of BIEs involving multiple obstacles that are close, touching, or even intersecting each other. Indeed, we demonstrate through numerical experiments that BIEs posed on the surface of composite obstacles, i.e., obstacles that can be expressed as unions of  geometrically simpler  intersecting obstacles, can be recast as BIEs posed on the union of the boundaries of the simpler domains, which, upon application of the proposed PWDI kernel-regularization technique, can be directly solved using BEM retaining the expected order of convergence. This aspect of the proposed technique may significantly simplify the numerical solution of many real-world problems involving intricate obstacles, as it effectively allows  bypassing the often involved task of meshing complex surfaces.

 The structure of this paper is as follows: The theoretical basis of Taylor interpolation on regular surfaces, and of density interpolation methods in general, are established in~Section~\ref{sec:pw_interp}.  Two PWDI procedures are next introduced in Section~\ref{sec:interpolating_functions}. One amounts to the non-trivial extension to three dimensions of the low-order closed-form analytic procedure put forth in~\cite{plane-wave:2018} (Section~\ref{Meq1}), while the other is a purely numerical procedure for the construction of arbitrarily high-order planewave density interpolants (Section~\ref{sec:higher_order}). Section~\ref{sec:nyst_bem} then provides the details on the discretization of kernel-regularized layer potentials and integral operators by means of a Chebyshev-based Nystr\"om method (Section~\ref{sec:Nystrom}) and a Galerkin BEM (Section~\ref{sec:BEM}). (The compatibility of the proposed approach with fast methods is addressed in Appendix~\ref{app:compatible}.) 
 Section~\ref{eq:numerical_results}, finally, presents a variety of numerical examples that validate and demonstrate the various capabilities of the PWDI technique in the context of both Nystr\"om and boundary element methods. 
 
\section{Preliminaries}\label{sec:prelim}


For the sake of definiteness we focus in this paper on scattering problems related to acoustic sound-soft and sound-hard scatterers, e.g., either Dirichlet or Neumann boundary conditions. We thus seek scattered fields that are solutions of the following exterior Dirichlet and Neumann boundary value problems
\begin{equation}\label{eq:Dirichlet}\lf\{\begin{split}
\Delta u^s_D+k^2 u^s_D=&~0\quad{\rm in}\ \mathbb{R}^3\setminus \Omega,\\
u^s_D+u^{\inc}=&~0\quad{\rm on}\ \Gamma,\\
\lim_{|\ner|\to\infty}|\ner|\lf(\frac{\partial u^s_D}{\partial |\ner|}-\im\!ku^s_D\rg)=&~0,
\end{split}\rg.\end{equation}
and 
\begin{equation}\label{eq:Neumann}\lf\{\begin{split}
\Delta u^s_N+k^2 u^s_N=&~0\quad{\rm in}\ \mathbb{R}^3\setminus \Omega,\\
\frac{\partial u^s_N}{\p \nor}+\frac{\partial u^{\inc}}{\p \nor}=&~0\quad{\rm on}\ \Gamma,\\
\lim_{|\ner|\to\infty}|\ner|\lf(\frac{\partial u^s_N}{\partial |\ner|}-\im\! ku^s_N\rg)=&~0,
\end{split}\rg.\end{equation}
respectively, where $\Omega\subset\R^3$ is a bounded obstacle whose boundary $\Gamma$ is a piecewise smooth, oriented and closed surface. (The incident fields $u^{\inc}$ in equations~\eqref{eq:Dirichlet} and~\eqref{eq:Neumann} are assumed to be solutions of the Helmholtz equation in all of $\R^3$.)

The  Dirichlet~\eqref{eq:Dirichlet} and Neumann~\eqref{eq:Neumann} scattering problems can be formulated via well-posed boundary integral equations by means of the combined field approach introduced by Brakhage-Werner~\cite{brakhage1965dirichletsche} and Burton-Miller~\cite{Burton1971Application}, respectively. The combined field approach relies on the use of Helmholtz single- and double-layer potentials, hereby denoted as
\begin{equation}\label{eq:LPots}
\lf(\mathcal S\varphi\rg) (\ner) := \int_{\Gamma} G(\ner,q)\varphi(q)\de s(q)\mbox{ and } \lf(\mathcal D\varphi\rg) (\ner) := \int_{\Gamma} \frac{\p G(\ner,q)}{\p \bnor(q)}\varphi(q)\de s(q),
\end{equation}
for  $\ner\in\R^3\setminus\Gamma$, respectively, where $G(\ner,\ner'):=(4\pi)^{-1}{\e^{\im\!k|\ner-\ner'|}}/{|\ner-\ner'|}\label{eq:GF}$
is the outgoing free-space Green function for the Helmholtz equation in~$\R^3$ with wavenumber~$k>0$.  (In what follows we utilize the symbol~$\ner$  to denote points that do not lie of the surface~$\Gamma$ while the symbols $p$ and $q$ are used exclusively to refer to points on the surface~$\Gamma$.)

Interior and exterior Dirichlet/Neumann traces of the single- and double-layer potentials give rise to the four boundary integral operators of the Calder\'on calculus associated with the Helmholtz equation. Specifically, the Helmholtz single-layer ($S$), double-layer ($K$), adjoint double-layer ($K'$) and hypersingular ($N$) operators are defined as
\begin{equation}\begin{split} \lf(S\varphi\rg)(p) := \int_{\Gamma} G(p,q)\varphi(q)\de s(q), &\qquad \lf(K'\varphi\rg)(p) := \int_{\Gamma} \frac{\p G(p,q)}{\p \bnor(p)}\varphi(q)\de s(q),\\
\lf(K\varphi\rg)(p) := \int_{\Gamma} \frac{\p G(p,q)}{\p \bnor(q)}\varphi(q)\de s(q),&\qquad
\lf(N\varphi\rg)(p) := \mathrm{f.p.}\int_{\Gamma} \frac{\p^2 G(p,q)}{\p \bnor(p)\p \bnor(q)}\varphi(q)\de s(q),\end{split}\label{eq:int_op}
\end{equation}
for  $p\in \Gamma$, where $\bold n(q)$ denotes the outward pointing unit normal to $\Gamma$ at $q\in\Gamma$. As usual, the initials f.p. in the definition of the hypersingular operator $N$ stand for Hadamard finite-part integral.  

The combined field approach consists of looking for a scattered field $u^s=u^s_D$ (resp. $u^s=u^s_N$)  in the form
\begin{equation}
u^s(\ner)=(\mathcal{D}\varphi)(\ner)-{\rm i}\eta (\mathcal{S}\varphi)(\ner)\ \ \ner\in\mathbb{R}^3\setminus\Gamma.\label{eq:layer_pot}
\end{equation}
where $\varphi=\varphi_D:\Gamma\to\C$ (resp. $\varphi=\varphi_N:\Gamma\to\C$) is an unknown density function and $\eta\in\R$ is the coupling parameter. 
The enforcement of Dirichlet and Neumann  boundary conditions on $\Gamma$ leads to the following combined field boundary integral equations
\begin{equation}\label{BW}
{\rm(BW)}\quad \frac{1}{2}\varphi_D(p)+(K\varphi_D)(p)-{\rm i}\eta(S\varphi_D)(p)=-u^{\inc}(p),\quad p\in\Gamma,
\end{equation}
and respectively
\begin{equation}\label{BM}
{\rm(BM)}\quad \frac{{\rm i}\eta}{2}\varphi_N(p)-{\rm i}\eta(K'\varphi_N)(p)+(N\varphi_N)(p)=-\frac{\partial u^{\inc}(p)}{\partial\bnor(p)},\quad p\in\Gamma.
\end{equation}
Both boundary integral equations~\eqref{BW} and~\eqref{BM} are well-posed in appropriate functional spaces provided that $\eta\in\mathbb{R}, \eta\neq 0$~\cite{COLTON:1983}. 

As is well known, one of the main challenges in the numerical discretization of boundary integral equations~\eqref{BW} and~\eqref{BM} is posed by the singular character of the kernels of the boundary integral operators defined in equations~\eqref{eq:int_op} as the integration point $q$ approaches the target point $p$. Indeed, for a sufficiently regular surface $\Gamma\subset\R^3$ the operators $S$, $K$, and $K'$ feature kernels with weak (integrable) singularities of type~$\mathcal{O}(|p-q|^{-1})$, while the operators $N$ feature hyper-singular kernels of type~$\mathcal{O}(|p-q|^{-3})$ as $\Gamma\ni q\to p\in\Gamma$. The numerical evaluation of the layer potentials~\eqref{eq:LPots}, on the other hand, faces the significant challenge of dealing with the nearly singular character of the integral kernels at points $\ner\in\R^3\setminus\Gamma$ lying near the boundary at which, although smooth, the integrands exhibit large derivatives that ultimately hinder the accuracy of standard integration procedures.

In what follows we present a density interpolation method aimed at expressing the boundary integral operators~\eqref{eq:int_op} and layer potentials~\eqref{eq:LPots} in terms of surface integrands of prescribed regularity. For presentation simplicity and without loss of generality, instead of treating each one of the integral operators~\eqref{eq:int_op} and layer potentials~\eqref{eq:LPots} separately, we focus on the combined field integral operators of the BW~\eqref{BW} and BM~\eqref{BM} integral equations and the associated combined field potential~\eqref{eq:layer_pot}.

\section{Kernel regularization via density interpolation\label{sec:pw_interp}}

Before we briefly embark on the presentation of the proposed density interpolation method, we first state some useful results of the differential geometry of surfaces that will provided the theoretical basis and the notation for the derivations presented below in this section. The main result of the next section is summarized in Remark~\ref{rem:taylor}.

\subsection{Taylor series on smooth surfaces}\label{sec:taylor}
We assume throughout this section that $\Gamma$ is a regular surface. First, given a system of coordinates around $p\in\Gamma$ with $\mathbf{x}(x_1,x_2)=p,\ \mathbf{x}:V\subset\mathbb{R}^2\to\Gamma$, we define the covariant basis of the tangent space $T_p\Gamma$ of $\Gamma$ at a point $p$ as
\begin{equation*}
\mathbf{e}_i(p):=\frac{\partial \mathbf{x}}{\partial x_i}(p),\quad i=1,2.
\end{equation*}
We will follow the usual convention of not using the argument $p$ whenever there is no possibility of confusion. Using the Riemannian metric tensor 
\begin{equation}
g_{ij}:=\langle\mathbf{e}_i,\mathbf{e}_j\rangle_p=\mathbf{e}_i\cdot\mathbf{e}_j,\quad 1\leq i,j\leq 2,
\label{eq:metric_tensor}\end{equation} 
we define the contravariant basis as
$\mathbf{e}^i:=\sum_{j=1}^2g^{ij}\mathbf{e}_j$, $i=1,2,$ in terms of the inverse of the metric tensor $(g^{ij})=(g_{ij})^{-1}$. We also denote by $g$ the determinant of the metric tensor $(g_{ij})$, that is $g=g_{11}g_{22}-g_{12}^2$. With these notations in place, we have that the unit normal at $p\in\Gamma$ is given by $\bnor={\mathbf{e}_1\wedge \mathbf{e}_2}/{\sqrt{g}}.$

Given a function $\varphi:\Gamma\to\mathbb{C}$ we define its tangential gradient (or contravariant gradient) by the formula
$\langle d\varphi,X\rangle_p={\rm grad}\ \varphi\cdot X$, for\ all  $X\in T_p\Gamma,$
where  $d\varphi=\partial_1\varphi\ \mathbf{e}_1+\partial_2\varphi\ \mathbf{e}_2$ is a $1-$form. An explicit formula for ${\rm grad}\ \varphi$ is given by
\begin{equation}\label{eq:grad_surf}
{\rm grad}\ \varphi=(g^{11}\partial_1\varphi+g^{12}\partial_2\varphi)\mathbf{e}_1+(g^{21}\partial_1\varphi+g^{22}\partial_2\varphi)\mathbf{e}_2=\partial_1\varphi\ \mathbf{e}^1+\partial_2\varphi\ \mathbf{e}^2.
\end{equation}
We also define the Hessian of $\varphi$, ${\rm Hess}(\varphi)$ at $p\in\Gamma$ as the linear operator
\[
{\rm Hess}(\varphi):T_p\Gamma\to T_p\Gamma,\quad {\rm Hess}(\varphi)(Y)=\nabla_Y{\rm grad}\ \varphi,\ Y\in T_p\Gamma,
\]
where $\nabla$ is the Riemannian connection on $\Gamma$. The latter can be expressed as 
\[
\nabla_{\mathbf{e}_i}\mathbf{e}_j=\sum_{\ell=1}^2\Gamma^{\ell}_{ik}\mathbf{e}_\ell,
\]
in terms of the Christoffel symbols defined by $\displaystyle\Gamma^\ell_{ij}:=\frac{\partial \mathbf{e}_i}{\partial x_j}\cdot \mathbf{e}^\ell.$

It can be shown that ${\rm Hess}(\varphi)$ can be also viewed as a symmetric bilinear form on $T_p\Gamma$ given by
${\rm Hess}(\varphi)(X,Y)=\langle{\rm Hess}(\varphi)X,Y\rangle_p$,  $X,Y\in  T_p\Gamma$.
The expression of ${\rm Hess}(\varphi)$ can be computed explicitly in the form
\begin{equation}\label{eq:hess}
{\rm Hess}(\varphi)=\sum_{i,j=1}^2\left(\partial_{i}\p_{j}\varphi-\sum_{\ell=1}^2\Gamma^\ell_{ij}\partial_\ell\varphi\right)\mathbf{e}^i\otimes\mathbf{e}^j,
\end{equation}
where $\mathbf{e}^i\otimes\mathbf{e}^j=\mathbf{e}^i\ (\mathbf{e}^j)^\top$. For a scalar function $\varphi:\Gamma\to\mathbb{C}$ and a multi-index $\alpha=(\alpha_1,\alpha_2),$ $\alpha_j\in\mathbb{Z}$, $\alpha_j\geq 0$, $j=1,2$, we denote by
\begin{equation}\label{eq:HO_der}
\partial^\alpha\varphi:=\partial_1^{\alpha_1}\partial_2^{\alpha_2}\varphi=\frac{\partial^{|\alpha|}\varphi}{\partial {x_1}^{\alpha_1}\partial{x_2}^{\alpha_2}}
\end{equation}
where $|\alpha|=\alpha_1+\alpha_2$.  

Finally, we need to make use of the \emph{exponential map} on $\Gamma$. This map is defined on an open neighborhood $\mathcal{U}$ of the origin in $T_p\Gamma$, that is $\exp_p:\mathcal{U}\subset T_p\Gamma\to \Gamma$ such that, for $v\in\mathcal{U}$ with $|v|$ small enough, $\exp_p(v)$ is defined as the point on $\Gamma$ which is distance $|v|$ away on the geodesic originating at $p$ and having velocity $v/|v|$ at $p$. With these notations in place, we are in the position to state \emph{Taylor's formula} in the form
\begin{equation}\label{eq:Taylor0}
\varphi(\exp_p(v))=\varphi(p)+v^\top{\rm grad}\ \varphi(p)+\frac{1}{2}v^\top{\rm Hess}(\varphi)(p)\ v+\mathcal{O}(|v|^3),\ p\in\Gamma,\ v\in T_p(\Gamma)
\end{equation}
as $|v|\to 0$, or equivalently as 
\begin{equation}\label{eq:Taylor}
\varphi(q)=\varphi(p)+v^\top{\rm grad}\ \varphi(p)+\frac{1}{2}v^\top{\rm Hess}(\varphi)(p)\ v+\mathcal{O}(|p-q|^3),\ p,q\in\Gamma,
\end{equation}
as $|p-q|\to 0$, where $v=\exp_p^{-1}(q)\in T_p\Gamma$  in the case when $\varphi$ is a smooth function defined on $\Gamma$. Taylor's formula can be carried to higher-order terms in the form
\begin{equation}\label{eq:Taylor_ho}
\varphi(q)=\varphi(p)+\sum_{j=1}^M\nabla_\Gamma^j\varphi(p)[v\otimes\cdots\otimes v]+\mathcal{O}(|p-q|^{M+1}),\ q,p\in\Gamma,
\end{equation}
as $|p-q|\to 0$, where $v=\exp_p^{-1}(q)\in T_p\Gamma$. and where the $j$-th tensor $(\nabla_\Gamma^j)\varphi:\underbrace{T_p\Gamma\times\cdots\times T_p\Gamma}_{j\ {\rm times}}\to \mathbb{R}$ is defined recursively as
\[
\nabla_\Gamma^j\varphi(Y)=\nabla_Y(\nabla_\Gamma^{j-1}\varphi),\quad j\geq 2,\quad \nabla_\Gamma\varphi:={\rm grad}\ \varphi, \quad Y\in T_p\Gamma,
\]
in terms of the Riemannian connection $\nabla$ on $\Gamma$. Clearly, as the surface gradient~\eqref{eq:grad_surf} and the Hessian~\eqref{eq:hess}, the higher order terms  $\nabla_\Gamma^j\varphi$, $j\geq 2$, can be expressed as a linear combination of tensor products of the form $\mathbf{e}^{i_1}\otimes\cdots\otimes\mathbf{e}^{i_j},\ i_\ell\in\{1,2\}$, whose coefficients, in turn, can be expressed as a linear combination of $\partial^\alpha\varphi$ for all $|\alpha|\leq j$. 

\begin{remark} \label{rem:taylor}The main take away message of this section is that Taylor's formula~\eqref{eq:Taylor_ho} implies that if two smooth density functions, say $\varphi$ and $\psi$, are such that  $\p^\alpha \varphi(p)=\p^\alpha\psi(p)$ for some $p\in\Gamma$ and for all $|\alpha|\leq M$, where the derivatives are taken with respect to any local parametrization of the surface around the point $p$, then $\varphi(p)=\psi(q) +\mathcal O(|p-q|^{M+1})$ as $\Gamma\ni q\to p\in\Gamma$. In the next section we will use that result produce a suitable Taylor-like interpolation of the density that will be used to regularize the boundary integrals. \end{remark}

\subsection{Kernel-regularized boundary integral operators and layer potentials}
Our density interpolation method relies on use of certain families of smooth functions $\Phi:\R^3\times\Gamma\to\C$ that are solutions of the Helmholtz equation
\[
\Delta_{\ner}\Phi(\ner,p)+k^2\Phi(\ner,p)=0,\quad \ner\in\mathbb{R}^3\quad{\rm for\ all}\ p\in\Gamma.
\]
Letting
\begin{equation}\label{eq:traces}
\Phi(q,p) := \lim_{\varepsilon\to 0}\Phi(q+\varepsilon \bnor(q),p)\andtext \Phi_n(q,p) := \lim_{\varepsilon\to 0}\nabla \Phi(q+\varepsilon \bnor(q),p)\cdot \bnor(q),
\end{equation}
 for any given $p\in\Gamma$, denote the Dirichlet and Neumann traces of such functions, respectively, we have that an application of the Green's third identity~\cite{COLTON:1983,NEDELEC:2001} leads to 
\begin{equation}
\boldsymbol 1_{\Omega}(\ner) \Phi(\ner,p) = -\int_{\Gamma}\frac{\partial G(\ner,q)}{\partial\bnor(q)}\Phi(q,p)\de s(q)+\int_\Gamma G(\ner,q)\Phi_n(q,p)\de s(q)\label{eq:green_PW}
\end{equation}
for all $\ner\in\R^3\setminus\Gamma$ and $p\in\Gamma$, where $\bol 1_{\Omega}$ denotes the characteristic function of the domain $\Omega$, i.e., $\bol 1_{\Omega}=1$ in $\Omega$ and $\bol 1_{\Omega}=0$ in~$\R^3\setminus\overline\Omega$. Therefore, combining the layer potential~\eqref{eq:layer_pot} with formula~\eqref{eq:green_PW}  we obtain the following equivalent expression for the combined field potential~\eqref{eq:layer_pot}:
\begin{equation}\label{eq:green_field_hoss}\begin{split}
u^s(\ner)=-\boldsymbol 1_{\Omega}(\ner)\Phi(\ner,p) + \int_{\Gamma}\frac{\partial G(\ner,q)}{\partial\bnor(q)}\lf\{\varphi(q)-\Phi(q,p)\rg\}\de s(q)\\
-\int_\Gamma G(\ner,q)\lf\{{\rm i}\eta\varphi(q)-\Phi_n(q,p)\rg\}\de s(q),\end{split}
\end{equation}
which is valid for all $\ner\in\mathbb{R}^3\setminus\Gamma$ and $p\in\Gamma$.

Letting then $\ner=p+\varepsilon\bnor(p)$, $\varepsilon>0$, and taking the limit of both sides of equation~\eqref{eq:green_field_hoss} as $\varepsilon\to 0^+$ we obtain the following reformulation of the BW boundary integral equation~\eqref{BW}
\begin{eqnarray}
\frac{1}{2}\lf\{\varphi(p)-\Phi(p,p)\rg\}+ \int_{\Gamma}\frac{\partial G(p,q)}{\partial\bnor(q)}\lf\{\varphi(q)-\Phi(q,p)\rg\}\de s(q)&\nonumber\\
-\int_\Gamma G(p,q)\lf\{{\rm i}\eta\varphi(q)-\Phi_n(q,p)\rg\}\de s(q)&~=-u^{\inc}(p)\quad\mbox{for all } p\in\Gamma,\label{eq:BW_hoss}
\end{eqnarray}
where we have utilized the standard jump conditions of the single- and double-layer operators~\cite{COLTON:1983,NEDELEC:2001}. 

The scope of the proposed density interpolation technique is to explicitly and efficiently construct a family of functions $\Phi(\ner,p)$ such that the integrands that enter in equation~\eqref{eq:BW_hoss} are regular (at least bounded) as $\Gamma\ni q\to p\in\Gamma$. To this end, for a given $\eta\in\mathbb{R},\eta\neq 0$, and a scalar function $\varphi:\Gamma\to\mathbb{C}$ which is assumed to be $(M+1)$-times continuously differentiable at $p\in\Gamma$,  we say that a family of functions~$\Phi(\ner,p)$ defined above satisfies  \emph{Taylor-like interpolation conditions of order $M\geq 0$} at $p\in\Gamma$ if its Dirichlet and Neumann traces defined in equations~\eqref{eq:traces}, satisfy 
\begin{subequations}\begin{equation}
\lim_{q\to p} \partial^{\alpha}\lf\{\varphi(q)-\Phi(q,p)\rg\} = 0\quad\mbox{for\ all}\quad |\alpha|\leq M,\andtext
\label{eq:cond_1_SL}\end{equation}
\begin{equation}
\lim_{q\to p} \partial^{\alpha}\lf\{{\rm i}\eta\varphi(q)-\Phi_n(q,p)\rg\} = 0\quad\mbox{for\ all}\quad |\alpha|\leq M,
\label{eq:cond_2_SL}\end{equation}\label{eq:cond_SL}\end{subequations}
respectively, where all the derivatives are taken with respect to $q$ on the surface. In the light of the Taylor's formula~\eqref{eq:Taylor_ho}, it is clear that 
\begin{equation}\label{eq:control_ho}
\lf|\varphi(q)-\Phi(q,p)\rg|\lesssim |q-p|^{M+1}\andtext \lf|{\rm i}\eta\varphi(q)-\Phi_n(q,p)\rg|\lesssim |q-p|^{M+1},
\end{equation}
 hold for all $p\in\Gamma$ at which the Taylor-like interpolation conditions~\eqref{eq:cond_SL} are satisfied, regardless of the surface parametrization underlaying~\eqref{eq:cond_SL} (see Remark~\ref{rem:taylor}). These estimates imply, in turn, that
\begin{equation*}\label{eq:control_ho_kernel}
\left|\frac{\partial G(p,q)}{\partial\bnor(q)}\lf\{\varphi(q)-\Phi(q,p)\rg\}\right|\lesssim |q-p|^{M}\mbox{ and } \lf|G(p,q)\lf\{{\rm i}\eta\varphi(q)-\Phi_n(q,p)\rg\}\rg|\lesssim |q-p|^{M}.
\end{equation*}
Therefore, from the estimates above we conclude that the proposed procedure effectively regularizes the singularities of the kernels of the boundary integral operators in equation~\eqref{eq:BW_hoss} provided  $\Phi$ satisfies the Taylor-like interpolation conditions~\eqref{eq:cond_SL} for~$M\geq0$.

Similarly, taking the exterior normal derivative of the expression~\eqref{eq:green_field_hoss} we obtain that the BM integral equation~\eqref{BM} can be equivalently  expressed as
 \begin{equation}\label{eq:BM_hoss}\begin{split}
\frac{1}{2}\lf\{{\rm i}\eta\varphi(q)-\Phi_n(q,p)\rg\}+ \int_{\Gamma}\frac{\p^2 G(p,q)}{\p\bnor(p)\partial\bnor(q)}\lf\{\varphi(q)-\Phi(q,p)\rg\}\de s(q)&\\
-\int_\Gamma \frac{\p G(p,q)}{\p \bnor(p)}\lf\{{\rm i}\eta\varphi(q)-\Phi_n(q,p)\rg\}\de s(q)=-\frac{\partial u^{\inc}(p)}{\partial\bnor(p)}&\ \mbox{ for all } p\in\Gamma,\end{split}
\end{equation}
where the integrands satisfy
\begin{equation*}\label{eq:control_ho_kernel_2}\begin{split}
\left|\frac{\p^2 G(p,q)}{\p\bnor(p)\partial\bnor(q)}\lf\{\varphi(q)-\Phi(q,p)\rg\}\right|\lesssim& |q-p|^{M-2}\quad\mbox{and}\\ \lf|\frac{\p G(p,q)}{\p \bnor(p)}\lf\{{\rm i}\eta\varphi(q)-\Phi_n(q,p)\rg\}\rg|\lesssim& |q-p|^{M},\end{split}
\end{equation*}
and are at least bounded provided $\Phi$ satisfies the Taylor-like interpolation conditions~\eqref{eq:cond_SL} at $p\in\Gamma$ for $M\geq 2$.

Finally, we apply the proposed density interpolation technique to the combined field potential~\eqref{eq:layer_pot} at observation points $\ner\in\R^3\setminus\Gamma$ near the boundary~$\Gamma$. Letting $p=p^*={\rm arg}\min_{q\in\Gamma}|\ner-q|\in\Gamma$ in the formula~\eqref{eq:green_field_hoss} for the combined field potential,  we obtain that the corresponding integrands in~\eqref{eq:green_field_hoss} satisfy
\begin{equation*}
\left|\frac{\partial G(\ner,q)}{\partial\bnor(q)}\{\varphi(q)-\Phi(q,p^*)\}\right|\lesssim \frac{|q-p^*|^{M+1}}{|q-\ner|^2}\leq |q-p^*|^{M-1}
\end{equation*}
and
\begin{equation*}
 \lf|G(\ner,q)\{{\rm i}\eta\varphi(q)-\Phi_n(q,p^*)\}\rg|\lesssim\frac{|q-p^*|^{M+1}}{|q-\ner|}\leq |q-p^*|^M,
\end{equation*}  provided $\Phi:\R^3\times\Gamma\to\C$ interpolates~$\varphi$---in the sense of the conditions in~\eqref{eq:cond_SL}---at a nearly singular point $p=p^*\in\Gamma$. Clearly, for sufficiently large interpolation orders~$M$, not only the integrands vanish at $p^*$, also their derivatives do. Finally, we mention that kernel-regularized expressions for the gradient of the combined field potential can be obtained by direct differentiation of~\eqref{eq:green_field_hoss}. 

 As we will see in the next section (and in numerical results presented Section~\ref{eq:numerical_results}),  expressions  for the potential and its normal derivative stemming  from~\eqref{eq:green_field_hoss} with $p=p^*={\rm arg}\min_{q\in\Gamma}|\ner-q|$, can be exploited to produce kernel-regularized operators for problems involving multiple obstacles that are close or even intersecting each other.

\begin{remark}\label{rem:eval_sep} Derivations similar to the ones presented above  can be carried out to produce kernel-regularized expressions for all four integral operators of Calder\'on calculus~\eqref{eq:int_op}. In fact, such expressions for the double-layer and hypersingular operators are given by the left-hand-side of~\eqref{eq:BW_hoss} and~\eqref{eq:BM_hoss}, respectively, that result from setting $\eta=0$ where~$\Phi$ must satisfy~\eqref{eq:cond_SL} with the corresponding  $\eta=0$ value. Similarly,  kernel-regularized expressions for the single-layer and adjoint double-layer operators are given by the left-hand-side of~\eqref{eq:BW_hoss} and~\eqref{eq:BM_hoss}, respectively, that result from dividing them by $-\im\!\eta$ and taking the limit~$\eta\to\infty$ where $\Phi$ must satisfy~\eqref{eq:cond_SL}  with the corresponding  $(-\im\!\eta)^{-1}=0$ value. 
\end{remark}

\begin{remark} Maue's  formula~\cite[Theorem 2.23]{COLTON:1983} provides  an alternative expression for the hypersingular operator. In fact, for a sufficiently regular surface $\Gamma$ and density function $\varphi$, the hypersingular operator can be equivalently expressed as 
 \begin{equation}\label{eq:Maue}\begin{split}
(N\varphi)(p)
=&\,k^2\int_{\Gamma}G(p,q)(\bnor(p)\cdot\bnor(q))\varphi(q)\de s(q)\\
&+{\rm p.v.}\int_\Gamma  \vv{\curl}_{\Gamma}^{p}G(p,q)\cdot\vv{\curl}^q_{\Gamma}\varphi(q)\de s(q),\end{split}
\end{equation}
where tangential rotational operator at a point $p\in\Gamma$ is defined as $\vv{\curl}^p_\Gamma= -\bnor(p)\wedge {\rm grad}^p_\Gamma $ in terms of the surface gradient~\eqref{eq:grad_surf} at $p\in\Gamma$ which is here denoted as ${\rm grad}_\Gamma^p$. As usual the initials p.v. in front of the integral sign stand for principal value integral. Using Maue's formula~\eqref{eq:Maue} we hence obtain that the BM integral equation~\eqref{BM} can be alternatively expressed as 
 \begin{eqnarray}
\frac{1}{2}\lf\{{\rm i}\eta\varphi(q)-\Phi_n(q,p)\rg\}+k^2\int_{\Gamma} G(p,q) \bnor(p)\cdot\bnor(q)\lf\{\varphi(q)-\Phi(q,p)\rg\}\de s(q)\label{eq:BM_hoss_reg}\\
\hspace{-0.2cm}+ \int_{\Gamma}\vv{\curl}_{\Gamma}^p G(p,q)\cdot\vv{\curl}_\Gamma^q\lf\{\varphi(q)-\Phi(q,p)\rg\}\de s(q)-\int_\Gamma \frac{\p G(p,q)}{\p \bnor(p)}\lf\{{\rm i}\eta\varphi(q)-\Phi_n(q,p)\rg\}\de s(q)\nonumber\\
=-\frac{\partial u^{\inc}(p)}{\partial\bnor(p)}\quad \mbox{ for all } p\in\Gamma,\nonumber
\end{eqnarray}
where the most singular integrand satisfies 
$$\left|\vv{\curl}_\Gamma^p G(p,q)\cdot\vv{\curl}_\Gamma^q\{\varphi(q)-\Phi(q,p)\}\right|\lesssim |q-p|^{M-2},$$  and is at least bounded provided $M\geq 2$. Both~\eqref{eq:BM_hoss} and~\eqref{eq:BM_hoss_reg} forms of the BM integral equation are considered in the numerical examples presented in Section~\ref{eq:numerical_results} below.
\end{remark} 
\begin{remark}\label{rem:hyper}
Yet another expression for the hypersingular operator can be easily derived from~\eqref{eq:Maue} by ``moving" the operator $\vv{\curl}_\Gamma^p$  outside the surface integral, i.e., 
\begin{equation}\label{eq:move_der}
{\rm p.v.}\int_\Gamma  \vv{\curl}_{\Gamma}^{p}G(p,q)\cdot\vv{\curl}^q_{\Gamma}\varphi(q)\de s(q) =-\curl^p_\Gamma \int_\Gamma  G(p,q)\vv{\curl}^{q}_{\Gamma}\varphi(q)\de s(q),
\end{equation}
where the scalar rotational operator on the right-hand-side is defined as $\curl_\Gamma^p=\bnor(p)\cdot {\rm grad}_{\Gamma}^{p}$. It thus follows from~\eqref{eq:move_der} that the hypersingular operator can be evaluated by applying the proposed technique to the single-layer operator alone, although separate Taylor-like interpolants for each one of the three components of $\vv{\curl}_\Gamma^q\varphi(q)$ and $\bnor(q)\varphi(q)$ are needed. This approach is utilized in Section~\ref{eq:numerical_results} to produce accurate BEM discretizations of the BM integral equation using $M=0$ and~$1$.
\end{remark}

\subsection{Multiple-scattering approach to scattering by composite surfaces} \label{sec:composites} Let $\Omega_j$, $j=1,2$, be open and simply connected domains  with smooth boundaries $\Gamma_j=\p\Omega_j$. Suppose $\Omega\subset\R^3$ is given by the union $\Omega = \Omega_1\cup\Omega_2$ where $\Omega_1\cap\Omega_2\neq \emptyset$. For the sake of conciseness and simplicity we focus here on the exterior Dirichlet problem~\eqref{eq:Dirichlet} which we proceed to formulate as a multiple scattering problem encompassing the two obstacles $\Omega_1$ and $\Omega_2$. 
This formulation is advantageous in many practical applications where suitable discrete representations of the ``combined surface" $\Gamma=\p (\Omega_1\cup\Omega_2)$ (in terms of surfaces meshes in the case BEM or manifold representations in terms of coordinate patches in the case Nystr\"om methods) are difficult to produce, but separate discretizations of its component parts, $\Gamma_1=\p\Omega_1$ and $\Gamma_2=\p\Omega_2$, are easy to generate.  An important example in this regard are Van der Waals molecular surfaces which are given by union of a typically large number of spherical atoms~(e.g.~\cite{chen2010tmsmesh}).

Instead of considering  the BW integral equation~\eqref{BW} on  $\Gamma=\p(\Omega_1\cup\Omega_2)$, we  pose it on $\tilde\Gamma=\Gamma_1\cup\Gamma_2$.  Letting 
$\tilde\varphi:\tilde\Gamma\to\C$ be a density function we  look for the scattered field in the form of the combined field potential  
$u^s(\ner) = (\tilde{\mathcal D}\tilde\varphi)(\ner) - \im\!\eta(\tilde{\mathcal S}\tilde\varphi)(\ner)$ where~$\tilde{\mathcal D}$ and $\tilde{\mathcal S}$ are the double- and single-layer potentials in~\eqref{eq:LPots}, but defined in terms of  surface integrals over~$\tilde \Gamma$. The enforcement of the Dirichlet boundary condition on~$\tilde\Gamma$  yields the integral equation
\begin{equation}\label{eq:IE_sep}
\frac{1}{2}\lf(\begin{array}{c}\varphi_1\\\varphi_2\end{array}\rg)  + \lf(\begin{array}{cc} K_{11}-\im\!\eta S_{11}&K_{12}-\im\!\eta S_{12} \\
K_{21}-\im\!\eta S_{21} & K_{22} - \im\!\eta S_{22}
\end{array}\rg)\lf(\begin{array}{c}\varphi_1\\\varphi_2\end{array}\rg) = \lf(\begin{array}{c}f_1\\ f_2\end{array}\rg)
\end{equation}
on $\tilde \Gamma$ for the unknown density function $\tilde\varphi$, where $\varphi_j=\tilde\varphi|_{\Gamma_j}$ and $f_j|_{\Gamma_j} = -u^\inc|_{\Gamma_j}$, $j=1,2$. The  operators $S_{ij}$ and $K_{ij}$, $i,j=1,2,$ in~\eqref{eq:IE_sep} are the single- and double-layer operators in~\eqref{eq:int_op} but defined in terms of boundary integrals over $\Gamma_i$ and target points $p\in\Gamma_j$. 
Note that the Dirichlet data $(f_1,f_2)$ in~\eqref{eq:IE_sep} requires the incident field $u^\inc$ to be defined on $\tilde\Gamma\setminus\Gamma$. Typically, $u^\inc$ is given by an explicit expression that can be directly evaluated almost everywhere in $\R^3$ including $\tilde\Gamma$. If that is not the case,~$(f_1,f_2)$ can be defined by simply extending $u^\inc$ to $\tilde\Gamma\setminus\Gamma$ by zero. Uniqueness of solutions for the integral equation~\eqref{eq:IE_sep} is stablished in Appendix~\ref{app:uniqueness}.

In order to evaluate the integral operators in~\eqref{eq:IE_sep} we apply the proposed technique to each one of the operators involving integration over the closed  surfaces $\Gamma_1$ and $\Gamma_2$. This is achieved by regularizing $I/2+K_{ij}-\im\!\eta S_{ij}$ for $i=j$ as an integral operator acting on $\Gamma_i$ using~\eqref{eq:BW_hoss}, and regularizing $K_{i,j}-\im\eta S_{ij}$  for $i\neq j$ as a layer potential that involves integration over $\Gamma_i$ and evaluation at target points $\ner\in\Gamma_j$ using~\eqref{eq:green_field_hoss}. 

The effectiveness of this approach is demonstrated by numerical examples based on the BEM presented in Section~\ref{eq:numerical_results}. A more extensive study of the multiple-scattering approach to scattering by composite surface---that,  in particular,  will include the BM integral equation~\eqref{BM}---will be presented in a future contribution. 

In the next section we present an explicit construction of families of functions $\Phi$ such that the conditions~\eqref{eq:cond_SL} are satisfied for any given smooth function $\varphi$.

\section{Interpolating functions}\label{sec:interpolating_functions}
This section is devoted to the construction of the functions $\Phi$ introduced in the previous section. Specifically, we look for expressions of the kind
\begin{equation}
\Phi(\ner,p) :=\sum_{|\alpha|=0}^{M}\p^\alpha\varphi(p)\Phi^{(1)}_\alpha(\ner,p) + {\rm i}\eta\sum_{|\alpha|=0}^{M}\p^\alpha\varphi(p)\Phi^{(2)}_\alpha(\ner,p),\label{eq:interpolants}
\end{equation}
where the derivatives $\p^\alpha \varphi$ are defined in~\eqref{eq:HO_der} and where the expansion  functions $\Phi^{(1)}_\alpha$ and $\Phi^{(2)}_\alpha$  are taken to be linear combinations of planewaves:
\begin{equation}
\Phi^{(1)}_{\alpha}(\ner,p) :=\sum_{\ell=1}^{L} a_{\ell,\alpha}(p)W_\ell(\ner-p)\andtext \Phi^{(2)}_{\alpha}(\ner,p) :=\sum_{\ell=1}^{L} b_{\ell,\alpha}(p)W_\ell(\ner-p),\label{eq:LCPW}
\end{equation}
where $W_\ell(\ner-p)=\exp\lf\{{\rm i}k\bol d_\ell\cdot\lf(\ner-p\rg)\rg\}$, $\ell=1,\ldots,L$, have distinct directions $\bol d_\ell$ ($|\bol d_\ell|=1$). The directions $\bol d_\ell$ may depend on $p\in\Gamma$, in which case we make the dependence explicit in the notation $\bol d_\ell=\bold d_{\ell}(p)$. Clearly, any linear combination of the form~\eqref{eq:LCPW} amounts to a homogeneous solution of the Helmholtz equation in the variable $\ner$. 

 The following lemma, whose proof follows directly from Taylor's theorem in two-dimensions,  establishes simple point conditions on the traces of the expansion functions~\eqref{eq:LCPW} that guarantee that the interpolation requirements \eqref{eq:cond_SL} are satisfied.

\begin{lemma}\label{lem:point_cond}  Let $\Phi:\R^3\times\Gamma\to\C$ be given by~\eqref{eq:interpolants}, where $\Phi^{(j)}_\alpha:\R^3\times\Gamma\to\C$, $|\alpha|\leq M$, $j=1,2$, are the linear combinations of planewaves defined in~\eqref{eq:LCPW}. Then, sufficient conditions for $\Phi$ to satisfy~\eqref{eq:cond_SL} at $p\in\Gamma$, are that the Dirichlet trace $\Phi^{(j)}_{\alpha}(q,p)$ and the  Neumann trace $\Phi^{(j)}_{n,\alpha}(q,p)=\lim_{\epsilon\to 0}\nabla_{\ner} \Phi^{(j)}_{\alpha}(q+\epsilon \bnor(q),p)\cdot \bnor(q)$, where $q\in\Gamma$ and $j=1,2$, satisfy
\begin{subequations}\begin{equation}
\p^{\beta} \Phi^{(1)}_{\alpha}(p,p) =\lf\{\begin{array}{ccl} 1&\mbox{if}&\beta=\alpha,\\
0&\mbox{if}& \beta\neq\alpha,\end{array}\rg.\quad\qquad \p^{\beta} \Phi^{(1)}_{n,\alpha}(p,p) =0,\label{eq:point_conditions_V}
\end{equation}
\begin{equation}
\p^{\beta} \Phi^{(2)}_{\alpha}(p,p) =0\andtext \p^{\beta} \Phi^{(2)}_{n,\alpha}(p,p) =\lf\{\begin{array}{ccl} 1&\mbox{if}&\beta=\alpha,\\
0&\mbox{if}& \beta\neq\alpha,\end{array}\rg.\label{eq:point_conditions_U}
\end{equation}\label{eq:point_conditions}\end{subequations}
for all sub-indices $\beta\in\mathbb{Z}^2_+$ such that  $|\beta|\leq M$.
\end{lemma}

The following two sections address the problem of finding explicit expressions for~$\Phi^{(1)}_{\alpha}$ and $\Phi^{(2)}_{\alpha}$, $|\alpha|\leq M$, by utilizing the point conditions~\eqref{eq:point_conditions}. 

\subsection{Closed-form planewave expansion functions in the case $M=1$}\label{Meq1}

In this section we find closed-form expressions for the families of  functions~$\{\Phi^{(1)}_\alpha\}_{|\alpha|\leq M}$ and $\{\Phi^{(2)}_\alpha\}_{|\alpha|\leq M}$, defined in~\eqref{eq:LCPW}, whose traces satisfy  the requirements in Lemma~\ref{lem:point_cond} for the interpolation order $M=1$.

We thus search for functions~$\{\Phi^{(1)}_\alpha\}_{|\alpha|\leq 1}$ and $\{\Phi^{(2)}_\alpha\}_{|\alpha|\leq 1}$ that are linear combinations of planewaves whose directions $\bol d_\ell$ depend on $p\in\Gamma$, that is $W_\ell(\ner-p)=\exp\lf\{{\rm i}k\bold d_\ell(p)\cdot\lf(\ner-p\rg)\rg\}$, $\ell=1,\ldots,L$. The planewave directions $\bold d_\ell(p)$ are expressed in terms of the basis $\{{\bol\tau}_1(p),{\bol\tau}_2(p),\bnor(p)\}$, that is
\[
\bold d_\ell(p)=d_{\ell,1}{\bol\tau}_1(p)+d_{\ell,2}{\bol\tau}_2(p)+d_{\ell,3}\bnor(p),\quad |\bold d_\ell(p)|=1,
\]
where the unitary contravariant vectors ${\bol\tau}_j(p), j=1,2,$ are defined by
\begin{equation}\label{eq:tang_vecs}
{\bol\tau}_1 :=\sqrt{\frac{g}{g_{22}}}\mathbf{e}^1\andtext {\bol\tau}_2 :=\sqrt{\frac{g}{g_{11}}}\mathbf{e}^2,
\end{equation}
in terms of  the Riemann metric tensor $(g_{ij})$ in~\eqref{eq:metric_tensor} and its determinant $g$. Note that we dropped the dependence on $p$ of all the quantities in equations~\eqref{eq:tang_vecs}, as there is no risk of confusion.

 We begin by defining, for $p\in\Gamma$ and a direction $\bold d(p)$ (which may not be unitary), the following functions
\begin{equation}\label{eq:trig_fnc}
S(\ner,p,\bold d(p)) :=\sin\lf(k\bold d(p)\cdot(\ner-p)\rg)\mbox{ and }
C(\ner,p,\bold d(p)) :=\cos\lf(k\bold d(p)\cdot(\ner-p)\rg)
\end{equation}
for $\ner\in\mathbb{R}^3$. Clearly, for unitary directions $\bold d_1(p)$ and $\bold d_2(p)$, the products 
$$C\lf(\cdot,\cdot,\frac{\bold d_1(p)}{\sqrt{2}}\rg)C\lf(\cdot,\cdot,\frac{\bold d_2(p)}{\sqrt{2}}\rg),\quad C\lf(\cdot,\cdot,\frac{\bold d_1(p)}{\sqrt{2}}\rg)S\lf(\cdot,\cdot,\frac{\bold d_2(p)}{\sqrt{2}}\rg), \mbox{ and }$$
$$ S\lf(\cdot,\cdot,\frac{\bold d_1(p)}{\sqrt{2}}\rg)S\lf(\cdot,\cdot,\frac{\bold d_2(p)}{\sqrt{2}}\rg)$$ are linear combinations of planewaves of the form~\eqref{eq:LCPW} with (unitary) directions $\frac{1}{\sqrt{2}}(\bold d_1(p)+\bold d_2(p))$ and $\frac{1}{\sqrt{2}}(\bold d_1(p)-\bold d_2(p))$  provided $\bold d_1(p)\cdot \bold d_2(p)=0$.

The following lemma introduces the sought linear combinations of planewaves: 
\begin{lemma} Let $S$ and $C$ be the functions defined in~\eqref{eq:trig_fnc} and denotes by $L=-\mathbf{e}_1\cdot\p_1\bnor=\p^2_1\bnex\cdot \bnor$, $M=-\mathbf{e}_1\cdot\p_2\bnor=-\mathbf{e}_2\cdot\p_1\bnor=\p_1\p_2\bnex\cdot \bnor,$ and $N=-\mathbf{e}_2\cdot\p_2\bnor=\p^2_2\bnex\cdot\bnor$  the second fundamental form coefficients at $p\in\Gamma$. Then, the Dirichlet and Neumann traces of
\begin{subequations}\label{eq:U2}\begin{eqnarray}
\Phi^{(2)}_{(0,0)}(\ner,p) &:=&\frac{1}{k}S\lf(\ner,p,\bnor(p)\rg),\\
\Phi^{(2)}_{(1,0)}(\ner,p) &:=&\frac{2}{k^2}\sqrt{\frac{g_{22}}{g}}S\lf(\ner,p,\frac{\bnor(p)}{\sqrt{2}}\rg)S\lf(\ner,p,\frac{\btau_1(p)}{\sqrt{2}}\rg), \\
\Phi^{(2)}_{(0,1)}(\ner,p) &:=&\frac{2}{k^2}\sqrt{\frac{g_{11}}{g}}S\lf(\ner,p,\frac{\bnor(p)}{\sqrt{2}}\rg)S\lf(\ner,p,\frac{\btau_2(p)}{\sqrt{2}}\rg),
\end{eqnarray}\end{subequations}
and
\begin{subequations}\label{eq:U1}\begin{eqnarray}
\Phi^{(1)}_{(0,0)}(\ner,p) &:=&C\lf(\ner,p,\bnor(p)\rg),\\
\Phi^{(1)}_{(1,0)}(\ner,p) &:=&\frac{1}{k}\sqrt{\frac{g_{22}}{g}}S\lf(\ner,p,\btau_1(p)\rg)-\lf\{\frac{g_{12}M-g_{22}L}{g}\rg\}\Phi^{(2)}_{(1,0)}(\ner,p)-\\
&&\lf\{\frac{g_{12}N-g_{22}M}{g}\rg\}\Phi^{(2)}_{(0,1)}(\ner,p),\nonumber\\
\Phi^{(1)}_{(0,1)}(\ner,p) &:=&\frac{1}{k}\sqrt{\frac{g_{11}}{g}}S\lf(\ner,p,\btau_2(p)\rg)-\lf\{\frac{g_{12}L-g_{11}M}{g}\rg\}\Phi^{(2)}_{(1,0)}(\ner,p)-\\
&&\lf\{\frac{g_{12}M-g_{11}N}{g}\rg\}\Phi^{(2)}_{(0,1)}(\ner,p),\nonumber
\end{eqnarray}\end{subequations}
 satisfy the requirements~\eqref{eq:point_conditions_U} and~\eqref{eq:point_conditions_V} of Lemma~\ref{lem:point_cond}, respectively,  for $M=1$.
\end{lemma}
\begin{proof}
The proof follows directly from the computation of the tangential derivatives of $\Phi^{(1)}_\alpha$ and $\Phi^{(2)}_\alpha$ at $\ner=p\in\Gamma$ and the use of the identities
\cite{do2016differential}
\begin{equation}\begin{split}
\partial_1\bnor=&~-\lf(g^{11}L+g^{12}M)\mathbf{e}_1-(g^{21}L+g^{22}M\rg)\mathbf{e}_2,\\
 \partial_2\bnor=&~-\lf(g^{11}M+g^{12}N)\mathbf{e}_1-(g^{21}M+g^{22}N\rg)\mathbf{e}_2,
\end{split}\end{equation} 
 where $(g^{ij})$ denotes the inverse of the metric tensor $(g_{ij})$.
\end{proof}

The extension of the trigonometric ansatz technology utilized above to the construction of  closed-form families of functions $\{\Phi^{(1)}_\alpha\}_{|\alpha|\leq M}$ and $\{\Phi^{(2)}_\alpha\}_{|\alpha|\leq M}$ that satisfy conditions in~\eqref{eq:point_conditions} for $M\geq 2$ is challenging. In particular, additional distinct planewave directions ought to be  incorporated in the ansatz. Because of the aforementioned difficulties, we advocate for the algebraic approach presented in following section to construct high-order planewave expansion functions in the case $M\geq 2$.

\subsection{Higher-order planewave expansion functions} \label{sec:higher_order}

In this section we develop a purely algebraic algorithm to construct  expansion functions $\{\Phi^{(1)}_\alpha\}_{|\alpha|\leq M}$ and $\{\Phi^{(2)}_\alpha\}_{|\alpha|\leq M}$,  at  a given (regular) point $p\in\Gamma$.  Unlike the  analytical approach presented in the previous  section, we now select a collection of planewave directions $\{\bol d_\ell\}_{\ell=1}^{\ell=L}$,  that are independent of $p\in\Gamma$. While the desired interpolation order $M$ and the number $L$ of planewave directions are parameters in our algorithm, the planewave directions themselves can be selected either randomly or uniformly from the unit sphere in three dimensions.  

It is clear from Lemma~\ref{lem:point_cond} that in order to find the desired  expansion coefficients $\{a_{\ell,\alpha}\}_{\ell=1,|\alpha|\leq  M}^{\ell=L}$  (resp. $\{b_{\ell,\alpha}\}_{\ell=1,|\alpha|\leq  M}^{\ell=L}$) that determine the expansion functions  $\{\Phi^{(1)}_\alpha\}_{|\alpha|\leq M}$  (resp. $\{\Phi^{(2)}_\alpha\}_{|\alpha|\leq M}$) at~$p\in\Gamma,$ one has to impose the  $N=(M+1)(M+2)$ independent conditions~\eqref{eq:point_conditions_V} (resp.~\eqref{eq:point_conditions_U}) which have to be satisfied exactly. Consequently, a solvable linear system for the coefficients could be produced provided the number of planewave directions satisfies  $L\geq N$. In order to form such a linear system, we proceed to sort the $N/2$ indices $\alpha=(\alpha_1,\alpha_2)$ satisfying $\alpha_1+\alpha_2\leq M$ by introducing a bijective mapping $J:\{|\alpha|\leq M\}\to \{1,\ldots,N/2\}$.  Therefore, letting $\bol\delta_{j}$, $j=1,\ldots,N,$ denote the canonical vectors of $\R^N$ we have that  conditions~\eqref{eq:point_conditions_V} lead to the linear system 
\begin{equation}\label{eq:lin_sym1}
\boldsymbol{C}(p)\bold a_{j}(p) = \bol\delta_{j},\quad j=1,\ldots,N/2,
\end{equation}
for the coefficient vector~$\bold a_{j}=[a_{1,J^{-1}(j)},\ldots,a_{L,J^{-1}(j)}]^T\in\C^{L}$, while conditions~\eqref{eq:point_conditions_U} yield the system 
\begin{equation}\label{eq:lin_sym2}
\boldsymbol{C}(p)\bold b_{j}(p) = \bol\delta_{j+N/2},\quad j=1,\ldots,N/2,
\end{equation}
for the coefficient  vector $\bold b_{j}=[b_{1,J^{-1}(j)},\ldots,b_{L,J^{-1}(j)}]^T\in\C^{L}$, where $\boldsymbol{C}(p)$ is a $N\times L$ complex-valued matrix. Note that we have assumed in these derivations that the first $N/2$ rows of $\boldsymbol{C}(p)$ correspond to the conditions on the Dirichlet traces, while the remaining $N/2$ rows correspond to the conditions on the Neumann traces. 

In the case when $L>N$, the solution of the linear systems~\eqref{eq:lin_sym1} and~\eqref{eq:lin_sym2} must be understood in the least-squares sense. Letting $\boldsymbol{C}^\dagger(p) = \lf[{\bold c}^\dagger_{1}(p),\dots,{\bold c}^\dagger_{{N}}(p)\rg]\in\C^{L\times N}$ denote the Moore-Penrose pseudoinverse of $\boldsymbol{C}(p)$, that is $\boldsymbol{C}(p)\boldsymbol{C}^\dagger(p)=\bol I$ with  $\bol I\in \R^{N\times N}$ being the identity matrix and $\bol C(p)$ being full-rank, the unknown vectors $\bold a_{j}$ and $\bold b_{j}$ are respectively  given by 
\begin{equation}
\bold a_{j}(p) = {\bold c}^\dagger_{{j}}(p)\andtext \bold b_{j}(p) ={\bold c}^\dagger_{j+N/2}(p),\quad j=1,\dots,N/2.
\end{equation}
Note that under the assumption that $\boldsymbol{C}(p)$ is full-rank, $\bol C^{\dagger}(p)$ can be computed explicitly via the formula $\bol C^{\dagger}(p)=\bol C^*(p)\lf(\bol C(p)\bol C^*(p)\rg)^{-1}$, where $\bol C^*$ denotes the Hermitian transpose of $\bol C$.

We present in Appendix~\ref{app:explicit_M_3} a recursive approach to compute the entries of the matrices $\bol C(p)$ in the case $M=3$.

\begin{remark}
We have found difficult to prove whether for a given set of distinct planewave directions $\{\bol d_\ell\}_{\ell=1}^{\ell=L}$ the matrix $\bol C(p)\in\C^{N\times L}$ is full-rank for any $p\in\Gamma$ and will be left for future work. 
In practice, however, planewave directions selected from a ``uniform" spherical grid give rise to numerically invertible matrices $\bol C(p)\bol C^*(p)$. In detail, the  planewave directions for the construction of numerical PWDI interpolants used throughout  this paper are given by
$(\cos\theta_m\sin\phi_n,\sin\theta_m\sin\phi_n,\cos\phi_n)$ where $\theta_m=2\pi(m-1/2)/L_\theta$ for  $m=1,\ldots,L_\theta$ and $\phi_n=\pi (n-1/2)/L_\phi$ for $n=1,\ldots,L_\phi$, with $L=L_\theta\times L_\phi=2\times 2,4\times3,5\times4,6\times5$ for interpolation orders $M=0,1,2,3$, respectively. 

\end{remark}

\section{Numerical evaluation of integral operators and layer potentials}\label{sec:nyst_bem}
This section presents Nystr\"om and BEM discretization schemes based on standard quadrature rules, for the numerical evaluation of the kernel-regularized integral operators and layer potentials
associated to the combined field integral equations~\eqref{BW} and~\eqref{BM}. 
 
\subsection{Chebyshev-based Nystr\"om method}\label{sec:Nystrom}
We here briefly describe the 3D boundary integral equation method introduced in our previous contribution~\cite{HDI3D}. The surface $\Gamma$ is represented as the union $\Gamma=\bigcup_{k=1}^{N_p} \overline{\mathcal P^k}$  of non-overlapping  patches  $\mathcal P^k$, $k=1,\dots,N_p$, where $\mathcal P^k\cap\mathcal P^{l}=\emptyset$ if $k\neq l$. It is  assumed (throughout this section) that each surface patch $\mathcal P^k$ has associated a bijective~$\mathcal C^\infty$ coordinate map $\bnex^k:\mathcal H\mathcal\to \overline{\mathcal P^k}$,
\begin{equation}\label{eq:maps}
\bnex^k(\bxi) := \lf(x^k_1(\xi_1,\xi_2),x^k_2(\xi_1,\xi_2),x^k_3(\xi_1,\xi_2)\rg),\quad k=1,\ldots,N_p,\quad (\bxi=(\xi_1,\xi_2))
\end{equation} where $\mathcal H = [-1,1]\times [-1,1]\subset\R^2$. Furthermore, the coordinate maps~\eqref{eq:maps} are selected in such a way that the unit normal
$
\bnor^k(\bxi) = {\p_1\bnex^k(\bxi)\wedge  \p_2\bnex^k(\bxi)}/{|\p_1\bnex^j(\bxi)\wedge  \p_2\bnex^j(\bxi)|}
$
at the point $\bnex^k(\bxi)\in\mathcal P^k$ points outward to the surface $\Gamma$.  The surface integral of a sufficiently regular function $F:\Gamma\to\R$---such as the integrands in~\eqref{eq:BW_hoss},~\eqref{eq:BM_hoss} and~\eqref{eq:green_field_hoss}---can then be expressed as
$$
\int_{\Gamma}F(\nex)\de s = \sum_{k=1}^{N_p} \int_{\mathcal H}F\lf((\bnex^k(\bxi)\rg)|\p_1\bnex^k(\bxi)\wedge\p_2\bnex^k(\bxi)|\de \bxi=\int_{\mathcal H}f(\bxi)\de \bxi.
$$

In order to  numerically evaluated the integral above with high-precision, we utilize open Chebyshev grids in the parameter space $\mathcal{H}$. Accordingly, $\mathcal H$ is discretized  by means of the so-called  Fej\'er's first quadrature rules~\cite{davis2007methods} which  yields the approximation
\begin{equation}
\int_{\mathcal H}f(\bxi)\de \bxi \approx \sum_{i=1}^N\sum_{j=1}^N f(t_i,t_j)\omega_{i}\omega_j\label{eq:quad_rule},
\end{equation}
 where the quadrature points $t_j$ are the Chebyshev zero points 
\begin{equation}
t_j := \cos\lf(\vartheta_{j}\rg),\quad\vartheta_j := \frac{(2j-1)\pi}{2N},\quad j=1,\ldots,N,\label{eq:grid_points}
\end{equation}
and the Fej\'er quadrature weights are given by
\begin{equation}
 \omega_j:= \frac{2}{N}\lf(1-2\sum_{\ell=1}^{[N/2]}\frac{1}{4\ell^2-1}\cos(2\ell\vartheta_{j})\rg),\quad j=1,\ldots,N.\label{eq:weights}
\end{equation}

A key feature of this discretization scheme is that  the quadrature rule~\eqref{eq:quad_rule} yields spectral (super-algebraic) accuracy for integration of smooth $C^\infty(\mathcal H)$ functions. As expected, however, slower convergence rates are achieved for less regular integrands (such as surface density functions associated to problems of scattering by piecewise smooth obstacles). Yet another important feature of this discretization scheme is that all the partial derivatives of the coordinate maps $\bnex^k$, unit normals $\bnor^k,$ and functions  $\varphi(\bnex^k(\bxi))$ that are needed for the construction of the planewave density interpolant $\Phi:\R^3\times \Gamma\to\C$, can be efficiently and accurately computed at the grid points $\lf(t_i, t_j\rg)$, $1\leq i,j\leq N$,  by means FFT differentiation. More details can be found in~\cite{HDI3D}.

Finally, the proposed procedure for the numerical evaluation of the BW combined field operator $K-{\rm i}\eta S$, using the method of Section~\ref{sec:higher_order}, is summarized in Algorithm~\ref{alg:N_BW}. A completely analogous procedure can be followed for evaluation of the BM combined field operator $N-{\rm i}\eta K'$.

\begin{algorithm}
\caption{Nystr\"om evaluation of the forward map $(K-\im\!\eta S)\varphi$}
\begin{algorithmic}
\REQUIRE{Grids $\{\nex^k_{i,j}\}_{i,j=1}^{i,j=N}\subset\mathcal P^k$, $k=1,\ldots, N_p$, corresponding to the discretization of the surface $\Gamma$ using $N_p$ non-overlapping patches, generated using Chebyshev grids in the parameters space $\mathcal H$;  discrete density function $\varphi(\nex_{i,j}^k)= \phi_{i,j}^k$, $i,j=1,\ldots,N.$ $k=1,\ldots, N_p$; planewave interpolation order $M$; planewave directions $\bol d_\ell,$ $\ell=1,\ldots,L$.}
\FOR{$k$ from 1 to $N_p$}
\STATE{Compute $\p^\alpha\varphi$ of all orders $|\alpha|\leq M$ on the patch $\mathcal P^k$ using FFT-based spectral differentiation of the 2D array $\{\phi^k_{i,j}\}_{i,j=1}^{i,j=N}$\;}
\ENDFOR
\STATE{Set $I_{i,j}^k=0$ for $i,j=1,\ldots,N$ and $k=1,\ldots,N_p$}
\FOR{each grid point $\nex_{i,j}^k$}
\STATE{Generate the coefficients $a_{\ell,\alpha_r}(\nex_{i,j}^k)$ and $b_{\ell,\alpha_r}(\nex_{i,j}^k)$ for $\ell=1,\ldots,L$ and $r=1,\ldots,(M+1)(M+2)/2$}
\STATE{Compute  the interpolating function $\Phi$~\eqref{eq:interpolants} using the derivatives $\p^\alpha\varphi$ at $\nex_{i,j}^k$ and the coefficients $a_{\ell,\alpha_r}$ and $b_{\ell,\alpha_r}$}
\FOR{$m$ from 1 to $N_p$}
\STATE{Evaluate the approximate integral $I=\sum_{\nex_{p,q}^m\in \mathcal P^m} f(\nex_{p,q}^m)w^m_{p,q}\approx \int_{\mathcal P^m}f(\ney)\de s$ with $f(\ney)=-\frac{\varphi(\nex_{i,j}^k)}{2}+\frac{\p G(\nex_{i,j}^k,\ney)}{\p n(\ney)}\{\varphi(\ney)-\Phi(\ney,\nex_{i,j}^k)\}-G(\nex_{i,j}^k,\ney)\lf\{{\rm i}\eta\varphi(\ney)-\p_n\Phi(\ney,\nex_{i,j}^k)\rg\}$ using Fej\'er's quadrature rule}
\STATE{Update $I_{i,j}^k=I_{i,j}^k + I$}
\ENDFOR
\ENDFOR

\RETURN $I_{i,j}^k$ for $i,j=1,\ldots N$ and $k=1,\ldots,N_p$.
\end{algorithmic}\label{alg:N_BW}
\end{algorithm}

 \subsection{Galerkin boundary element method}\label{sec:BEM}
This section concerns the use of the proposed planewave density interpolation method within the standard second-order BEM discretization using triangular surface meshes. 

To fix ideas, we consider once again the BW boundary integral equation~\eqref{BW} which upon constructing an  appropriate density interpolation function $\Phi:\R^3\times\Gamma\to\C$, can be equivalently expressed (in strong form) as~\eqref{eq:BW_hoss}.  Throughout this section we assume that  $\Omega\subset\R^3$ is a bounded Lipschitz polyhedral domain. Therefore, associated to the surface $\Gamma=\p\Omega$ there is a triangulation $\mathcal T_h$ such that $\Gamma = \overline {\bigcup_{T\in\mathcal T_h}T}$. Note that both the single- and double-layer operators are bounded on $H^{1/2}(\Gamma)$ and that the Green's identities used in the derivation of~\eqref{eq:BW_hoss} still hold true for $\Gamma$ being Lipschitz~\cite{Mclean2000Strongly}. 
Therefore,  assuming that $u^\inc|_{\Gamma}\in H^{1/2}(\Gamma)$ we readily have that the variational formulation of~\eqref{eq:BW_hoss}---or equivalently~\eqref{BW}---reads as: Find $\varphi\in H^{1/2}(\Gamma)$ such that: 
\begin{equation}\label{eq:VF_BW}\begin{split}
 \int_{\Gamma}\lf\{\frac{1}{2}\{\varphi(p)-\Phi(p,p)\}+\int_{\Gamma}\frac{\p G(p,q)}{\p \bnor(q)}\{\varphi(q)-\Phi(q,p)\}\de s(q)\rg\}\psi(p)\de s(p)\\-\int_{\Gamma}\lf\{\int_{\Gamma}G(p,q)\{\im\!\eta\varphi(q)-\Phi_n(q,p)\}\de s(q)\rg\}\psi(p)\de s(p) = -\int_\Gamma u^\inc(p)\psi(p)\de s(p),
\end{split}
\end{equation}
for all test functions $\psi\in H^{-1/2}(\Gamma)$\footnote{We assume here the continuous extension of the standard real pairing $(u,v) = \int_{\Gamma}uv \de s$ for $u,v\in L^2(\Gamma)$ to the dual pairing $\langle\cdot,\cdot\rangle_{H^{1/2}(\Gamma)\times H^{-1/2}(\Gamma)}$.}. 
\begin{figure}[h!]\centering
\includegraphics[scale=1]{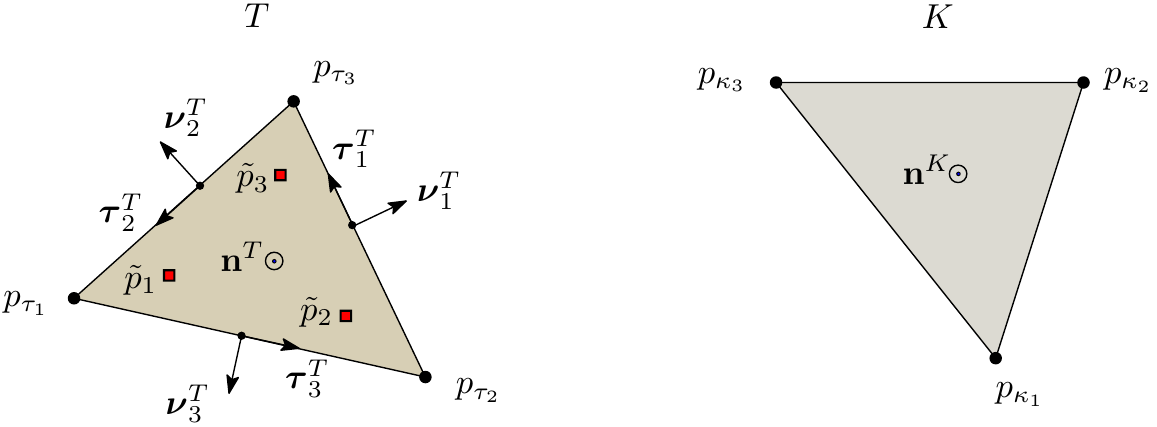}
\caption{Notation used for the the mesh triangles in the outer ($T$) and inner ($K$) integrals in~\eqref{eq:discrete_integral} and in the Algorithm~\ref{alg:BEM_BW}.}\label{fig:triangles}
\end{figure}
In order to find an approximate weak solution $\varphi_h\in H^{1/2}(\Gamma)$ of~\eqref{eq:VF_BW}, we resort to a Galerkin BEM  for which we consider the (finite-dimensional) subspace $W_h=\{v\in C(\Gamma): v|_{T}\mbox{ is a linear function }\forall T\in \mathcal T_h\}\subset H^{1/2}(\Gamma)\subset H^{-1/2}(\Gamma)$. Clearly, the set  $\{v_j\}_{j=1}^N$ of linear  polynomials  supported on  $\bigcup_{p_j\in T}T$  satisfying the condition $v_j(p_i)=\delta_{i,j}$ for all $i,j=1\ldots,N$, where $p_j$, $j=1,\ldots,N$ are the mesh nodes, forms a basis of the subspace~$W_h$. In detail, letting $p_{\tau_1},p_{\tau_2}$ and $p_{\tau_3}$ with $\tau_1,\tau_2,\tau_3\in \{1,\ldots,N\},$ denote the vertices of a triangle $T\in\mathcal T_h$, and defining the unit vectors (see Figure~\ref{fig:triangles})
$$
\bol\tau^T_1 = \frac{p_{\tau_3}-p_{\tau_2}}{|p_{\tau_3}-p_{\tau_2}|},\qquad \bol\tau^T_2 = \frac{p_{\tau_1}-p_{\tau_3}}{|p_{\tau_1}-p_{\tau_3}|},\qquad \bol\tau^T_3 = \frac{p_{\tau_2}-p_{\tau_1}}{|p_{\tau_2}-p_{\tau_1}|},
$$
$$
\bol\nu^T_1 = \bol\tau^T_1\wedge\bnor^T,\qquad \bol\nu^T_2 = \bol\tau^T_2\wedge\bnor^T,\andtext \bol\nu^T_3 = \bol\tau^T_3\wedge\bnor^T,
$$
we have that the sought approximate density function $\varphi_h\in W_h$ is given by 
\begin{equation}\label{eq:galerkin_approx}
\varphi_h(p) = \sum_{j=1}^3\varphi_h(p_{\tau_j})v_{\tau_j}(p),\quad p\in T,
\end{equation}
where the basis functions are
\begin{equation}\label{eq:basis_functions}
v_{\tau_j}(p) = 1-\frac{(p-p_{\tau_j})\cdot \bol \nu^T_j}{h^T_j}, \quad j=1,2,3,\quad p\in T,
\end{equation}
with $h_j = (p_{\tau_i}-p_{\tau_j})\cdot \bol \nu^T_j>0,$ $i\neq j$. The discrete variational formulation can thus be expressed as: Find $\varphi_h\in  W_h$ such that: 
\begin{eqnarray}
 \int_{\Gamma}v_j(p)\lf\{\frac{1}{2}\{\varphi_h(p)-u(p,p)\}+\int_{\Gamma}\frac{\p G(p,q)}{\p \bnor(q)}\{\varphi_h(q)-\Phi(q,p)\}\de s(q)\hspace{0.5cm}\rg.\label{eq:DVF}\\
-\lf.\int_{\Gamma}v_j(p)G(p,q)\{\im\!\eta\varphi_h(q)-\Phi_n(q,p)\}\de s(q)\rg\}\de s(p) = -\int_\Gamma v_j(p)u^\inc(p)\de s(p),\nonumber
\end{eqnarray}
for all basis functions  $v_j\in W_h$, $j=1,\ldots,N$. 

Consider now the term in~\eqref{eq:DVF} associated with the single-layer operator; that is
\begin{eqnarray}\int_{\Gamma} v_j(p)\lf\{\int_{\Gamma} G(p,q)\{\im\!\eta\varphi_h(q)-\Phi_n(q,p)\}\de s(q) \rg\}\de s(p)\hspace{3cm}\nonumber\\
=\sum_{T\in\mathcal T_h}\sum_{K\in\mathcal T_h}\int_{T} v_j(p)\lf\{\int_{K} G(p,q)\{\im\!\eta\varphi_h(q)-\Phi_n(q,p)\}\de s(q) \rg\}\de s(p). 
\label{eq:SLpart}\end{eqnarray} Clearly, both the outer and inner integrals in~\eqref{eq:SLpart} have to be evaluated by means of quadrature rules for which typically the vertices of  the triangules $T\in\mathcal T_h$ are used as quadrature points. Suppose the vertices $p_{\tau_\ell}$, $\ell=1,2,3$, are used as quadrature points for evaluation of the outer integral (over $T$). In order to regularize the integrand of the inner integral (over $K$) then, the normal derivative $\Phi_n$ of the planewave interpolant $\Phi$ has to approximate to sufficiently high-order the density function $\im\!\eta\varphi_h$ at the vertices  $p_{\tau_\ell}$, $\ell=1,2,3$. Unfortunately, the problem with this numerical integration scheme is that the construction of~$\Phi$ requieres the surface unit normal and surface tangent vectors to be properly defined at the vertices $p_{\tau_\ell}$, $\ell=1,2,3,$ which does not typically happen for general polyhedral surfaces.

In order to circumvent this issue we propose to use of a second-order Gauss quadrature rule for triangles that makes use of quadrature points $\tilde p_\ell$, $\ell=1,2,3,$ that lie in the interior $T^\circ$ of the triangle $T$~\cite{cowper1973gaussian,dunavant1985high}. This quadrature rule yields the approximation 
\begin{equation}\begin{split}
\int_{T} v_j(p)\lf\{\int_K G(p,q)\{\im\!\eta\varphi_h(q)-\Phi_n(q,p)\}\de s(q) \rg\}\de s(p) \hspace{3cm}\\
\approx\frac{|T|}{3}\sum_{\ell=1}^3  v_j(\tilde p_\ell)\lf\{\int_{K} G(\tilde p_\ell,q)\{\im\!\eta\varphi_h(q)-\Phi_n(q,\tilde p_\ell)\}\de s(q)\rg\},
\end{split}\label{eq:semi_discrete}\end{equation}
where the quadrature points  are given by \begin{eqnarray}
\tilde p_1= \frac{2p_{\tau_1}}{3}+\frac{p_{\tau_2}}{6}+\frac{p_{\tau_3}}{6},\  \
\tilde p_2=\frac{p_{\tau_1}}{6}+\frac{2p_{\tau_2}}{3}+\frac{p_{\tau_3}}{6},\  \ 
\tilde p_3= \frac{p_{\tau_1}}{6}+\frac{p_{\tau_2}}{6}+\frac{2p_{\tau_3}}{3}.\hspace{-0.5cm}
\label{eq:QP}\end{eqnarray} 
Since the quadrature points $\tilde p_\ell$, $\ell=1,2,3$, lie in the interior of the triangle $T$ (see Figure~\ref{fig:triangles}),  the surface unit normal and the surface tangent vectors---which are required in the construction of the planewave interpolant---are uniquely defined at those points. 

The inner integral, on the other hand, can be approximated by means of any sufficiently high-order quadrature rule. Using the standard node-based quadrature rule, for instance, we obtain
\begin{equation}\begin{split}
\int_{T} v_j(p)\lf\{\int_{K} G(p,q)\{\im\!\eta\varphi_h(q)-\Phi_n(q,p)\}\de s(q) \rg\}\de s(p) \hspace{3cm}\\
\approx\frac{|T||K|}{9}\sum_{\ell=1}^3\sum_{m=1}^3  v_j(\tilde p_\ell) G(\tilde p_\ell,p_{\kappa_m})\{\im\!\eta\varphi_h(p_{\kappa_m})-\Phi_n(p_{\kappa_m},\tilde p_\ell)\},
\end{split}\label{eq:discrete_integral}\end{equation}
where  $p_{\kappa_m}$, $m=1,2,3,$ are  the vertices of the triangle $K$ (see Figure~\ref{fig:triangles}). A completely analogous approach can be followed to evaluate the term  in~\eqref{eq:DVF} involving the double-layer operator. 

\begin{remark}Note that the terms inside the double sum on the right-hand side of~\eqref{eq:discrete_integral} are always well defined even when the triangles $T$ and $K$ coincide. The effect of the PWDI technique lies then in the regularization of the nearly-singular integral kernels that arise when the triangles $T$ and $K$ coincide or are close to each other. In the latter case, however, the effectiveness of the proposed technique is affected by the limited (piecewise planar) global regularity assumed on the surface parametrization. It is thus not worth to pursue interpolation orders $M\geq 2$ in the context of the proposed BEM for piecewise planar surface representations. \end{remark}

In what follows we describe in some detail the construction of the closed-form planewave interpolant for $M=1$ in the BEM context. As discussed in Section~\ref{Meq1} above, the construction of the planewave interpolant $\Phi:\R^3\times\Gamma\to\C$ requires the knowledge of a local smooth parametrization of the surface $\Gamma$ at and around the interpolation point $p\in\Gamma$. By construction, the interpolation point $p$ lies always in the interior of some triangle $T\in\mathcal T_h$. Therefore, the local surface parametrization has constant tangent vectors $\bold e^T_1$ and $\bold e^T_2$ that can be computed directly from the node data. In fact, selecting $\bold e^T_1 = \bol \nu^T_1$ and $\bold e^T_2 = \bol\tau^T_1$,  for instance, we have that the expressions for the planewave interpolants~\eqref{eq:U2}-\eqref{eq:U1} simplify significantly due to the fact that metric tensor becomes the identity, i.e., ~$g_{i,j} = \delta_{i,j}$, and the second fundamental form coefficients vanish, i.e., $L=M=N=0$, at $p\in T^\circ$. 

Having defined the tangent vectors $\bold e^T_1$ and $\bold e^T_2$, the surface derivatives of $\varphi_h\in W_h$ at quadrature points $\tilde p_\ell\in T^\circ$, $\ell=1,2,3,$ can be computed by direct differentiation of~\eqref{eq:galerkin_approx} which in turn involves the derivatives 
\begin{equation}\label{eq:surf_grad}
\p^\alpha v_{\tau_j}(\tilde p_\ell) = \lf\{\begin{array}{ccc}\displaystyle\frac{1+3\delta_{j,\ell}}{6}&\mbox{if} &\alpha =(0,0),\smallskip\\
\displaystyle -\frac{\bold e^T_1\cdot \bol\nu^T_j}{h^T_j}&\mbox{if}  &\alpha = (1,0),\smallskip\\
\displaystyle  -\frac{\bold e^T_2\cdot \bol\nu^T_j}{h^T_j}&\mbox{if}  &\alpha = (0,1),\smallskip\\
  0&\mbox{if}  &|\alpha|>1.
 \end{array}\right.\qquad  
\end{equation}
of the basis functions~\eqref{eq:basis_functions}.
An algorithmic description of the numerical evaluation of $(v_j,(K-\im\eta S)\varphi_h)$ for $j=1,\ldots,N$, is given in~Algorithm~\ref{alg:BEM_BW}.

\begin{algorithm}
\caption{BEM evaluation of  $( v_j,(K-{\rm i}\eta S)\varphi_h)$ for $j=1,\ldots, N$}
\begin{algorithmic}
\REQUIRE{Triangular  mesh $\mathcal T_h$ of the surface $\Gamma\subset\R^3$ consisting of $M$ triangles and $N$ mesh nodes  $\{p_j\}_{j=1}^N\subset\Gamma$; coefficients $\{\varphi_j\}_{j=1}^{N}\subset\C$ of the density function $\varphi_h(p)= \sum_{j=1}^N\varphi_{j}v_j(p)$, $p\in\Gamma$, with respect to the basis $\{v_j\}_{j=1}^N$ of piecewise linear polynomials.}
\STATE{Set $I_j=0$ for all $j=1,\ldots,N$}
\FOR{$\tau$ from 1 to $M$}
\STATE{Compute $\bold e^T_1,\bold e^T_2, \bold n^T=\bold e_1^T\wedge \bold e_2^T$ and the area $|T|$ of the $\tau$-th triangle $T$ with vertex indices $\{\tau_1,\tau_2,\tau_3\}\subset\{1,\ldots,N\}$}
\STATE{Produce $\tilde p_{1}, \tilde p_{1}$ and $\tilde p_{3}$ from the vertices $p_{\tau_1}$, $p_{\tau_2}$ and $p_{\tau_2}$ using~\eqref{eq:QP}}
\STATE{Compute $\p^{\alpha}\varphi_h$, $|\alpha|\leq M$, at $\tilde p_\ell\in T$, $\ell=1,2,3$ using~\eqref{eq:surf_grad}, to construct the planewave interpolants $\Phi(\cdot,\tilde p_\ell),$  $\ell=1,2,3$}
\FOR{$\kappa$ from 1 to $M$}
\STATE{Compute $\bold n^K$ and the area $|K|$ of the $\kappa$-th mesh triangle $K$ with vertex indices $\{\kappa_1,\kappa_2,\kappa_3\}\subset\{1,\ldots,N\}$}
\FOR{$m$ from 1 to $3$}
\STATE{Evaluate $F_\ell = f(\tilde p_\ell,p_{\kappa_m})$, $\ell=1,2,3,$  where
$f(p,q) =-\frac{\varphi_h(p)}{2}+ \frac{\p G( p,q)}{\p \bnor(q)}\{\varphi_h(q)-\Phi(q,p)\}-G( p,q)\{\im\!\eta\varphi_h(q)-\Phi_n(q, p)\}$.  The normal derivatives $\frac{\p G( p,q)}{\p \bnor(q)}$ and $\Phi_n(q,p)=\frac{\p \Phi(q,p)}{\p \bnor(q)}$ are computed with respect to the unit normal $\bnor^K$}
\STATE{Update $I_{\tau_1}= I_{\tau_1} + \frac{|T||K|}{9}\lf\{v_{\tau_1}(\tilde p_{1})F_1+v_{\tau_1}(\tilde p_{2})F_2+v_{\tau_1}(\tilde p_{3})F_3\rg\}$} 
\STATE{Update $I_{\tau_2}=I_{\tau_2} + \frac{|T||K|}{9}\lf\{v_{\tau_2}(\tilde p_{1})F_1+v_{\tau_2}(\tilde p_{2})F_2+v_{\tau_2}(\tilde p_{3})F_3\rg\}$} 
\STATE{Update $I_{\tau_2}= I_{\tau_3} + \frac{|T||K|}{9}\lf\{v_{\tau_3}(\tilde p_{1})F_1+v_{\tau_3}(\tilde p_{2})F_2+v_{\tau_3}(\tilde p_{3})F_3\rg\}$}
\ENDFOR
\ENDFOR
\ENDFOR
\RETURN $I_{j}$ for $j=1,\ldots,N$
\end{algorithmic}\label{alg:BEM_BW}
\end{algorithm}

In order to tackle the BM integral equation~\eqref{BM}, in turn, we apply the closed-form density interpolation technique to both single- and double-layer operators separately. In detail, we first resort to identities~\eqref{eq:Maue} and~\eqref{eq:move_der} to express the hypersingular operator in terms of single-layer operators. Then, upon integration by parts, the discrete variational formulation for the BM integral equation reads as: Find $\varphi_h\in W_h$ such that 
\begin{equation}
\frac{\im\!\eta}{2}\lf( v_j,\varphi_h\rg)-\im\!\eta\lf( v_j,K'\varphi_h\rg)-\lf( \vv{\curl}_\Gamma v_j,S\,\vv{\curl}_\Gamma\varphi_h\rg)
+k^2\lf( v_j\bnor,S\bnor\varphi_h\rg)=-\lf( v_j,\frac{\p u^\inc}{\p \bnor}\rg)\label{eq:VF_BM}
\end{equation}
for all basis functions $v_j\in W_h$, $j=1,\ldots,N$, where $(\cdot,\cdot)$ denotes the standard real pairing $(\varphi,\psi) =\int_{\Gamma}\varphi(q)\cdot \psi(q) \de s(q)$.  Noting that $( v_j,K'\varphi_h)=( Kv_j,\varphi_h)$ we hence conclude that it suffices to apply the proposed technique to both $S$ and $K$ separately (see Remark~\ref{rem:eval_sep}).

Finally, in order to produce accurate evaluations of  the combined field potential~\eqref{eq:layer_pot} at target points $\ner\in\R^3\setminus\Gamma$ near the surface $\Gamma$, we resort once again to the interior quadrature points~\eqref{eq:QP}. Indeed, in the context of the BEM the combined field potential at a point $\ner\in\R^3\setminus\Gamma$ can be expressed as
\begin{equation}\begin{split}
u_D^s(\ner) \approx-\boldsymbol 1_{\Omega}(\ner)\Phi(\ner,p^*) +\sum_{T\in\mathcal T_h} \int_{T}\frac{\partial G(\ner,q)}{\partial\bnor(q)}\lf\{\varphi_h(q)-\Phi(q,p^*)\rg\}\de s(q)-\hspace{0.7cm}\\
\sum_{T\in\mathcal T_h}\int_T G(\ner,q)\lf\{{\rm i}\eta\varphi_h(q)-\Phi_n(q,p^*)\rg\}\de s(q)\mbox{ with } p^*={\rm arg}\min_{q\in\Gamma}|\ner-q|,
\end{split}
\end{equation}
where the integrals over $T$ are approximated as follows
\begin{equation*}\begin{split}
 \int_{T}\frac{\partial G(\ner,q)}{\partial\bnor(q)}\lf\{\varphi_h(q)-\Phi(q,p^*)\rg\}\de s(q) \approx&  \frac{|T|}{3}\sum_{\ell=1}^3\frac{\partial G(\ner,\tilde p_\ell)}{\partial\bnor(\tilde p_\ell)}\lf\{\varphi_h(\tilde p_\ell)-u(\tilde p_\ell,p^*)\rg\},\\
\int_T G(\ner,q)\lf\{{\rm i}\eta\varphi_h(q)-\Phi_n(q,p^*)\rg\}\de s(q) \approx& \frac{|T|}{3}\sum_{\ell=1}^3 G(\ner,\tilde p_\ell)\lf\{{\rm i}\eta\varphi_h(\tilde p_\ell)-\Phi_n(\tilde p_\ell,p^*)\rg\}.
 \end{split}
\end{equation*}

 \section{Numerical examples}\label{eq:numerical_results}
This section presents a variety of numerical experiments that illustrate different aspects of the proposed methodology. 

 \subsection{Validation of the density interpolation procedures}
Our first numerical example is devoted to the validation of the two density interpolation procedures introduced above in Section~\ref{sec:interpolating_functions}. We start of by taking  $\Gamma$ as the (smooth) boundary of the bean-shaped obstacle displayed in Figure~\ref{fig:exp1a}, on which we define the density function
\begin{equation}
\rho(q) := \varphi(q)-\Phi(q,p^*),\quad q,p^*\in\Gamma, \label{eq:error_dens}
\end{equation}
where $\varphi$ is a given smooth density and  $\Phi$ is the planewave interpolant at $p^*\in\Gamma$.  Note that, by construction, $\rho$ and its  first $M$ tangential derivatives vanish at  $p^*\in\Gamma$. As was discussed in Section~\ref{sec:Nystrom}, the surface $\Gamma$ is here represented  by means of six non-overlapping rectangular patches each of which is discretized using Chebyshev grids consisting of $50\times 50$ points. Figure~\ref{fig:exp1a}  (resp. \ref{fig:exp2a}) displays the real part of  $\rho$ produced by the closed-form (resp. algebraic)  procedure. Figures~\ref{fig:exp1b} and~\ref{fig:exp1c} (resp.~\ref{fig:exp2b} and~\ref{fig:exp2c}), in turn, display slices of the real part of $\p^\alpha\rho$, $|\alpha|=M$, at $p^*$ in the parameter space, obtained using the closed-form (resp. algebraic) procedure with $M=1$ (resp. $M=3$).   The density function $\varphi$ utilized here is selected as the Dirichlet trace of the field produced by a point source at the point $\ner_0 = (0.1, -0.1,0.25)$ placed inside~$\Gamma$. The wavenumber and the coupling parameter considered in this example are $k=10$ and~$\eta =k$, respectively. These results demonstrate that the prescribed Taylor interpolation order $M$ is achieved by the proposed procedures. Similar results are obtained for the imaginary part of $ \rho$ as well as for $\rho_n(q) = {\rm i}\eta\varphi(q)-\Phi_n(q,p^*)$, which, for the sake brevity, are not displayed here.

 \begin{figure}[h!]
\centering	
\begin{subfigure}[b]{0.3\textwidth}{\centering\includegraphics[scale=0.76]{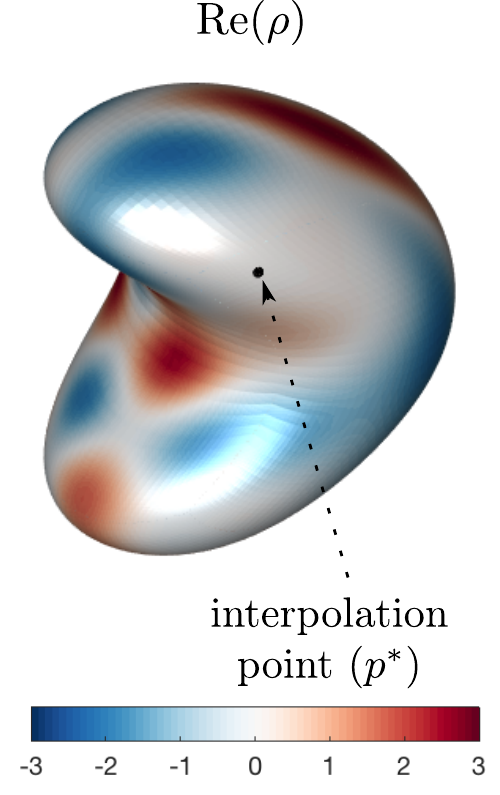}\caption{\ }\label{fig:exp1a}}\end{subfigure}
\begin{subfigure}[b]{0.3\textwidth}{\centering\includegraphics[scale=.6]{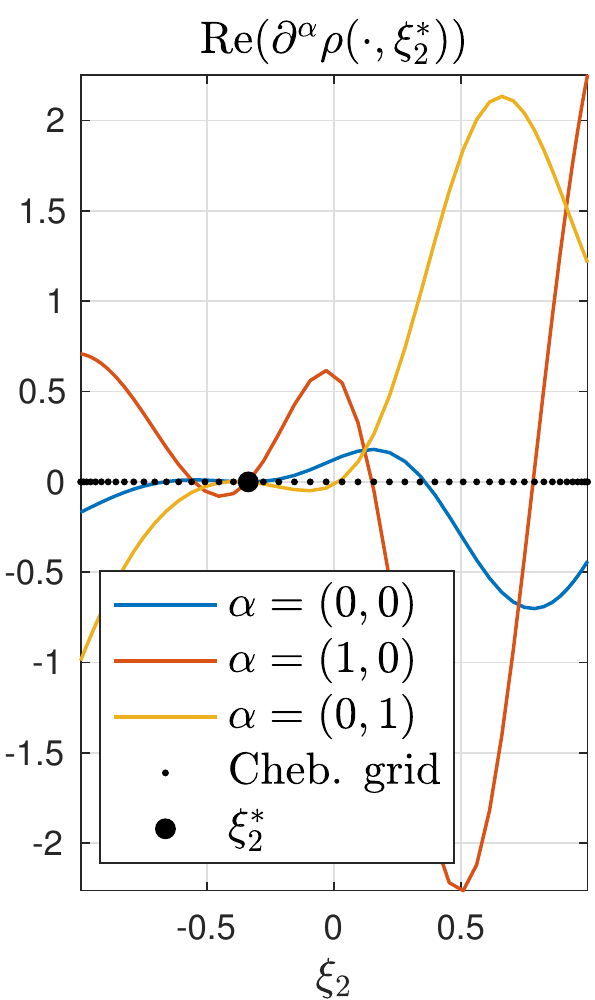}\caption{\ }\label{fig:exp1b}}\end{subfigure}
\begin{subfigure}[b]{0.3\textwidth}{\centering\includegraphics[scale=.6]{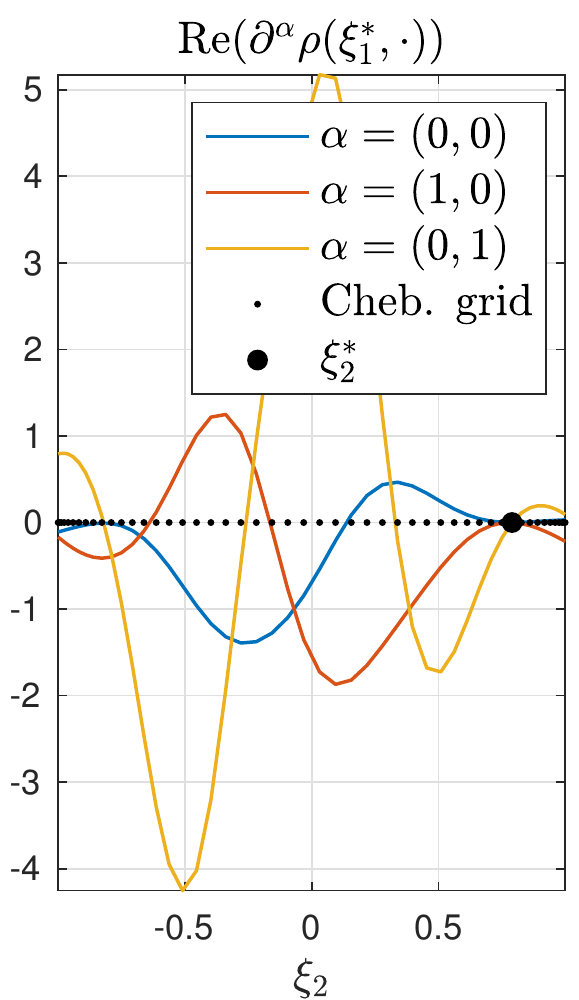}\caption{\ }\label{fig:exp1c}}\end{subfigure}\\
\begin{subfigure}[b]{0.3\textwidth}{\centering\includegraphics[scale=0.76]{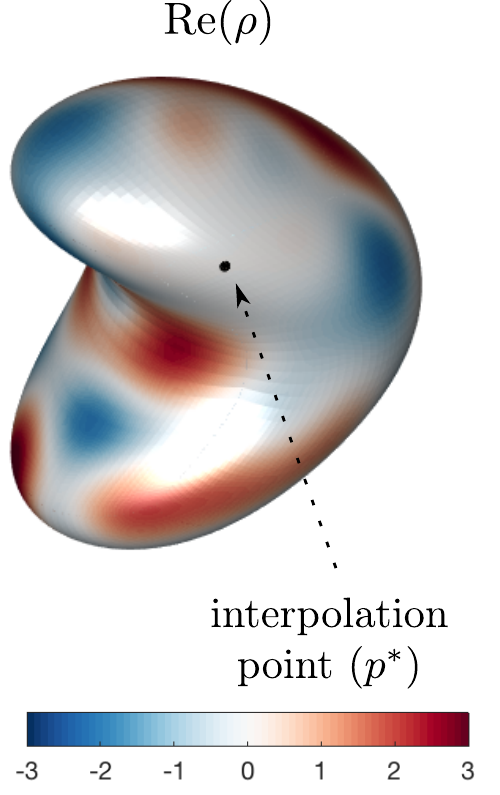}\caption{\ }\label{fig:exp2a}}\end{subfigure}
\begin{subfigure}[b]{0.3\textwidth}{\centering\includegraphics[scale=.6]{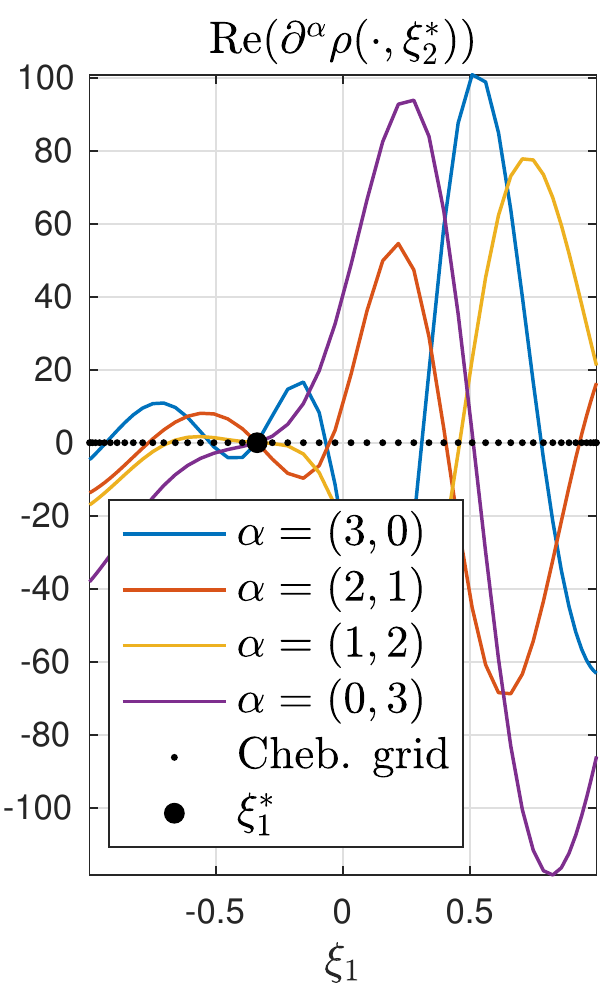}\caption{\ }\label{fig:exp2b}}\end{subfigure}
\begin{subfigure}[b]{0.3\textwidth}{\centering\includegraphics[scale=.6]{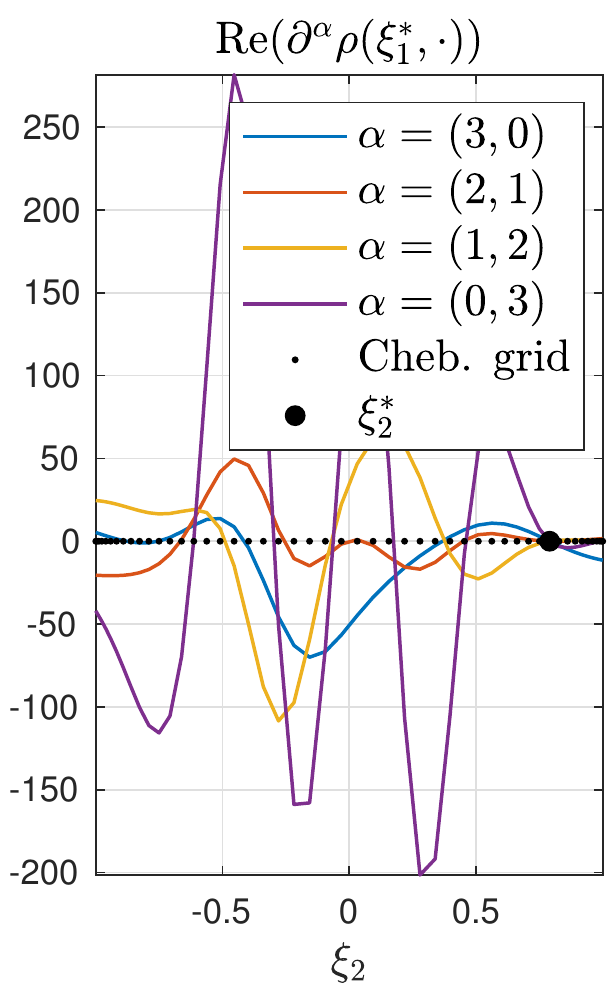}\caption{\ }\label{fig:exp2c}}\end{subfigure}
\caption{(a) (resp.~(d)): Plot of real and imaginary parts of $\rho$ defined in~\eqref{eq:error_dens} where the planewave interpolant $\Phi$ was constructed using the analytic (resp. numerical) procedure described in Section~\ref{Meq1} (resp.~\ref{sec:higher_order}). The interpolation point $p^*:=\bnex(\xi_1^*,\xi_2^*)=(-0.616, 0.310,0.599)$ is marked by a black dot. (b) and (c) (resp. (e) and~(f)): Plots of the cross section of the partial derivatives  $\p^{\alpha}\rho$ for all $|\alpha|= 1$ (resp. $|\alpha|=3$) in the parameter space. Note that all the first (resp. third) order derivatives vanish exactly at the interpolation points $(\xi_1^*,\xi_2^*)=(-0.339, 0.790)$.}\label{fig:reg_functions_analy}
\end{figure}
%
%

\subsection{Nystr\"om  and Boundary Element methods}

This section illustrates the capabilities of the density interpolation method for the regularization of the combined field potential and  associated BW and BM integral operators. 


 \begin{figure}[h!]
\centering	
\begin{subfigure}[b]{0.49\textwidth}{\centering\includegraphics[scale=.3]{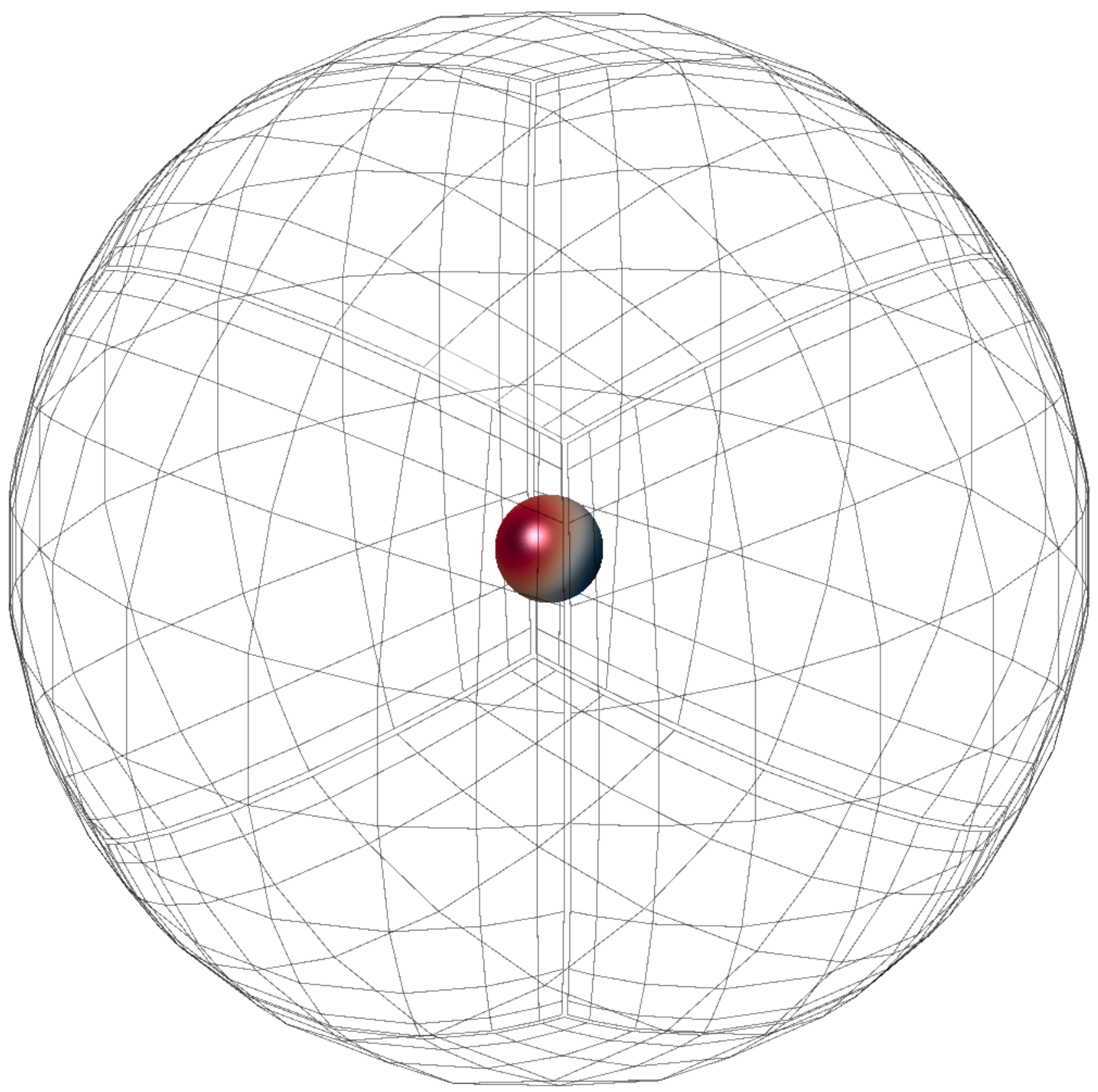}\caption{Far-field evaluation grid}\label{fig:exp3a}}\end{subfigure}
\begin{subfigure}[b]{0.49\textwidth}{\centering\includegraphics[scale=.22]{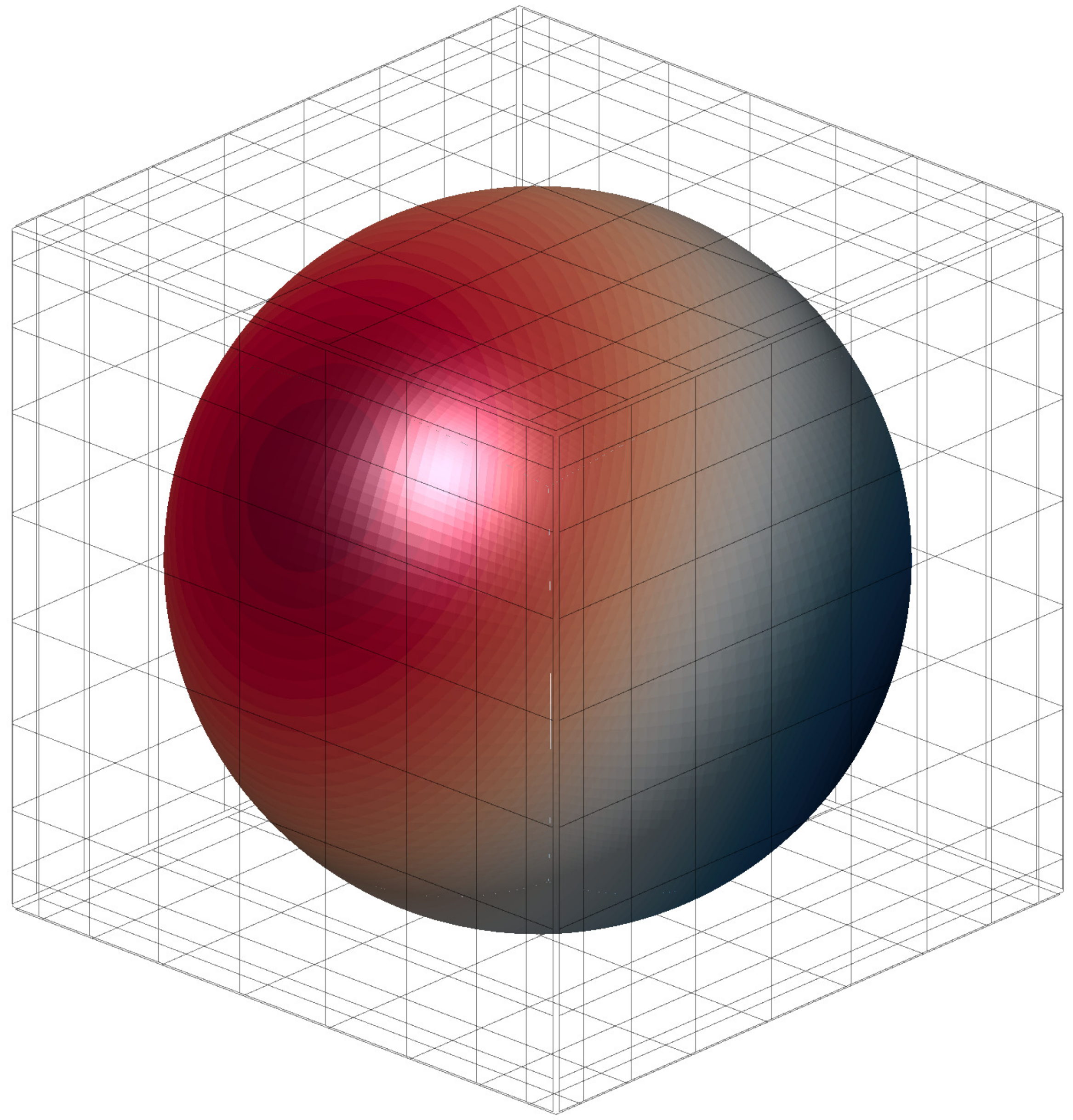}\caption{Near-field evaluation grid}\label{fig:exp3b}}\end{subfigure}
\caption{Grids utilized in the evaluation of the far- and near-field  errors.}\label{fig:FF_NF_grids}
\end{figure}

 \begin{figure}[h!]
\centering	
\begin{subfigure}[b]{0.49\textwidth}{\centering\includegraphics[scale=.55]{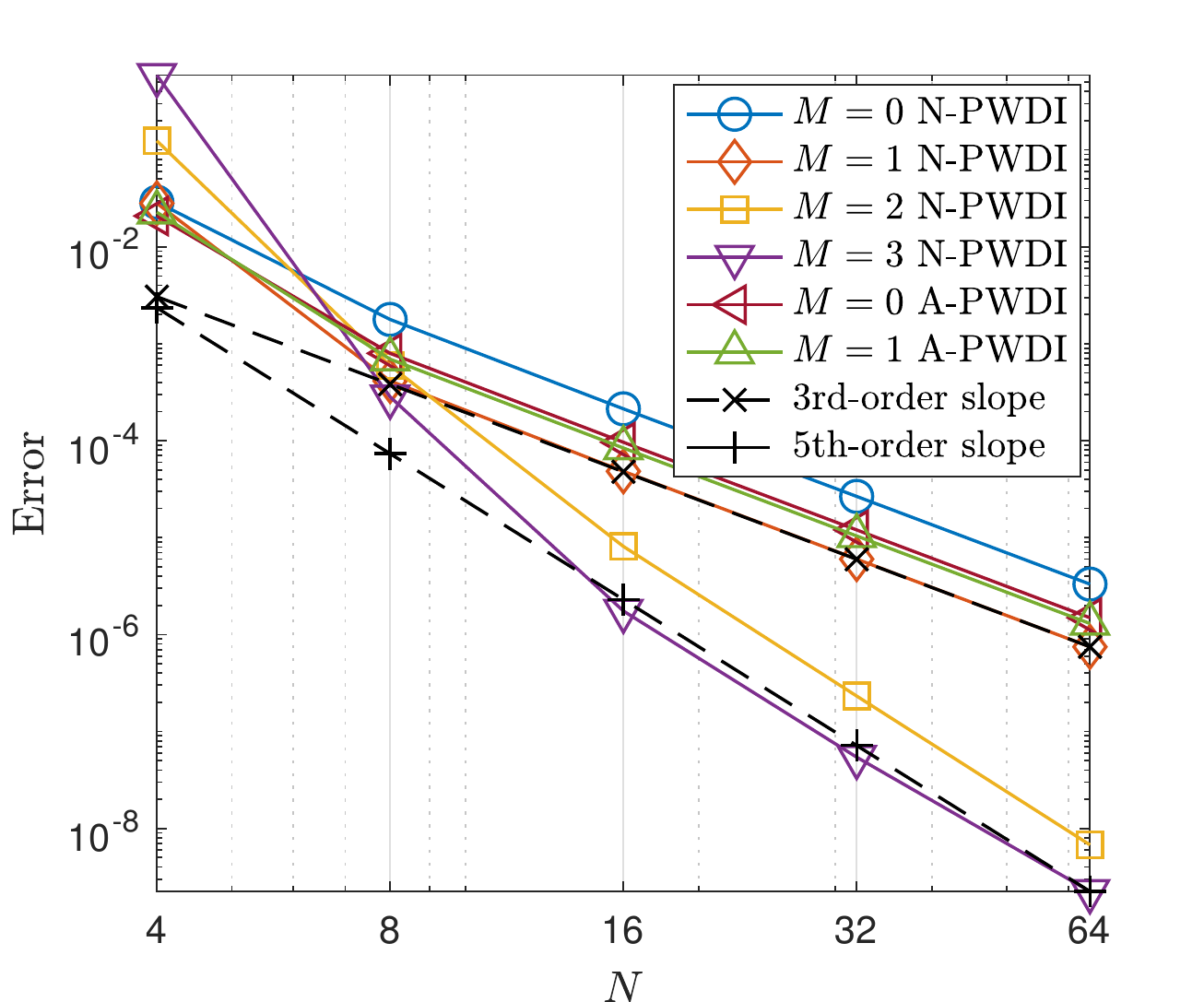}\caption{Far-field (relative) errors}\label{fig:exp3c}}\end{subfigure}
\begin{subfigure}[b]{0.49\textwidth}{\centering\includegraphics[scale=.55]{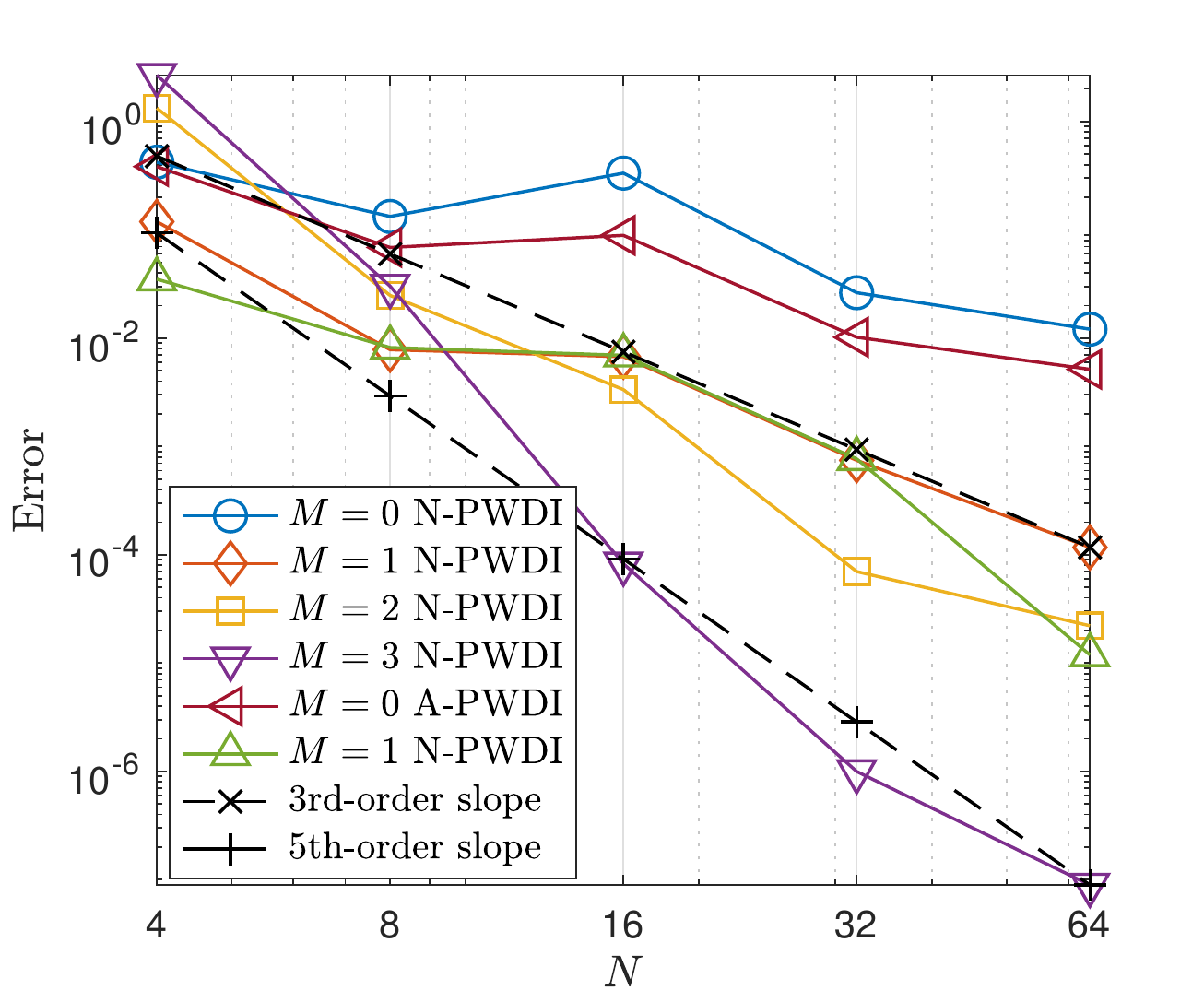}\caption{Near-field (relative) errors}\label{fig:exp3d}}\end{subfigure}
\caption{ Far- and near-field  errors in the solution of the Dirichlet problem~\eqref{eq:Dirichlet}  produced by the Nystr\"om method discretization of the BW integral equation~\eqref{BW} using the two proposed  PWDI techniques for different interpolation orders $M$ and grid sizes~$N$ (each quadrilateral surface patch is discretized using $N\times N$ quadrature points).}\label{fig:errors_sphere_BW}
\end{figure}

\subsection{Simple surfaces} In our first example we let $\Gamma$ be a unit sphere at the origin. (This simple surface has the advantage that can be easily represented using both quadrilateral patches and triangular meshes, allowing us to tackle the same problems using Nystr\"om and BEM methods.)  In order to assess the numerical errors, we consider an exact solution for both Dirichlet~\eqref{eq:Dirichlet} and Neumann~\eqref{eq:Neumann} problems, which is given by $u^s_{\rm exact}(\ner)= \e^{\im\! k|\ner-\ner_0|}/|\ner-\ner_0|-\e^{\im\! k|\ner-\ner_1|}/|\ner-\ner_1|$ where the source points $\ner_0 = (0.2, 0.1, 0.1)$ and $\ner_1=(-0.1, 0.3, -0.1)$ lie inside the unit sphere $\Gamma$. The real part of the Dirchlet trace of $u^s_{\rm exact}$ on $\Gamma$ (for $k=1$) is plotted  (in colors) in Figures~\ref{fig:exp3a} and~\ref{fig:exp3b}. The far-field errors in the  numerical solution $u^s$ are measured by
\begin{equation}\label{eq:error_formula}
{\rm Error} = \frac{\max_{\ner\in\Xi}|u^s_{\rm exact}(\ner)-u^s(\ner)|}{\max_{\ner\in\Xi}|u^s_{\rm exact}(\ner)|},
\end{equation} where $\Xi$ is the spherical grid (of radius $r=10$) displayed in Figure~\ref{fig:exp3a}.  The near-field errors, in turn, are  measured  using~\eqref{eq:error_formula} but with $\Xi$  being  the grid of the unit cube displayed in Figure~\ref{fig:exp3b}. Note that $\Gamma$ touches the cube surface grid at the center of its six faces.

The accuracy of the Chebyshev-based Nystr\"om method is assessed first. Figures~\ref{fig:exp3c} and~\ref{fig:exp3c}  display the far- and near-field errors, respectively, in the approximate Dirichlet solution obtained from the kernel-regularized BW integral equation~\eqref{eq:BW_hoss} for $k=\eta=1$. These figures display the errors obtained using the (closed-form) analytical  (A-PWDI) procedure as well as the (algebraic) numerical  (N-PWDI) density interpolation procedure introduced in Sections~\ref{Meq1} and~\ref{sec:higher_order}, respectively, for various discretization sizes~$N$ and density interpolation orders~$M$. The surface $\Gamma$ is here represented using six quadrilateral patches and each one of them is discretized using a Chebyshev grid of $N\times N$ points. 
The near fields, in particular, were computed using the kernel-regularized combined field potential~\eqref{eq:green_field_hoss}. As can be observed in these results,  the proposed technique yields third-order convergence of the far fields (as the grid size $N$ increases) for interpolation orders $M=0$~and~$1$ and fifth-order convergence for $M=2$~and~$3$.  In the near field, on the other hand,  third-order convergence is observed for all orders with the only exception of $M=3$ for which fifth order is achieved as $N$ increases. It is worth mentioning that the associated linear systems were solved iteratively by means of GMRES~\cite{saad1986gmres} with an error tolerance of~$10^{-8}$.  A nearly constant number of iterations ($\sim\!10$) was needed in all the examples considered in Figure~\ref{fig:errors_sphere_BW}. As in the case of the Laplace equation in 2D using the trapezoidal rule~\cite[Remark~6.1]{HDI3D}, the fact that the interpolation orders $M=0$~and~$1$, and also $M=2$~and~$3$, render the same order of convergence can be explained by the fact that Fej\'er's quadrature rule in this case integrates certain odd singular terms exactly. 

The  geometric setup of Figure~\ref{fig:FF_NF_grids} is next used to assess the accuracy of the  Chebyshev-based Nystr\"om method when dealing with the more challenging BM integral equation~\eqref{BM} for the solution of the Neumann problem~\eqref{eq:Neumann}.  The relevant numerical results are summarized in Table~\ref{tb:BM_results} where it can be clearly seen that, despite the overall smaller errors obtained for $M=3$, both interpolation orders $M=2$ and $M=3$ yield  far- and near-field errors that exhibit  the same nearly third-order convergence rate as $N$ increases. Direct~\eqref{eq:BM_hoss} and regularized~\eqref{eq:BM_hoss_reg} versions of the hyper-singular operator were considered in these examples. A number of GMRES iterations as large as 200 was needed to achieved the desired accuracy in some of the examples considered in Table~\ref{tb:BM_results} due to the known unfavorable spectral properties of the hypersingular operator present in the BM integral equation. As is well known, the number of GMRES iterations can be drastically reduced by considering appropriate preconditioners for the hypersingular operator~(cf.~\cite{boubendir2013wave,Bruno:2012dx}).

 \begin{table}
   \begin{center}
     \scalebox{0.95}{\begin{tabular}{c|c|c|c|c}
 \multicolumn{5}{c} {Nystr\"om method --- BM integral equation} \\     
\toprule
&  \multicolumn{2}{c|}{$M=2$}& \multicolumn{2}{c}{$M=3$} \\
\cline{2-5	}
$N$& \footnotesize{Direct}& \footnotesize{Regularized}& \footnotesize{Direct}& \footnotesize{Regularized}   \\
\cline{2-5}
 &\multicolumn{4}{c} {Far-field} \\
\hline
4& $8.55\cdot 10^{-2}$&$2.01\cdot 10^{-1}$&$8.89\cdot 10^{-2}$&$6.66\cdot 10^{-2}$ \\
8&$1.73\cdot 10^{-3}$ &$8.44\cdot 10^{-3}$&$3.55\cdot 10^{-4}$&$2.03\cdot 10^{-3}$\\
16& $3.46\cdot 10^{-4}$ &$9.46\cdot 10^{-4}$&$5.31\cdot 10^{-5}$&$2.60\cdot 10^{-4}$ \\
32&$4.13\cdot 10^{-5}$ &$1.21\cdot 10^{-4}$&$7.22\cdot 10^{-6}$&$3.22\cdot 10^{-5}$\\
 \hline
& \multicolumn{4}{c} {Near-field} \\
\hline
 4 &$1.61\cdot 10^{-0}$&$9.37\cdot 10^{-1} $&$1.19\cdot 10^{-0}$&$2.48\cdot 10^{-0}$\\
8&$7.32\cdot 10^{-3}$&$1.63\cdot 10^{-2}$&$3.31\cdot 10^{-3}$ &$6.72\cdot 10^{-3}$\\
16&  $1.86\cdot 10^{-3}$&$1.53\cdot 10^{-3}$&$1.79\cdot 10^{-4}$ &$8.35\cdot 10^{-4}$\\
32&$2.27\cdot 10^{-5}$ & $2.02\cdot 10^{-4}$&$2.08\cdot 10^{-5}$ &$8.29\cdot 10^{-5}$\\
 \bottomrule
\end{tabular}}
\caption{\label{tb:convergence_near_far}  Far- and near-field  errors in the solution of the Neumann problem~\eqref{eq:Neumann}  produced by the Nystr\"om  discretization of the BM integral equation~\eqref{BM} using the numerical PWDI procedure of Section~\ref{sec:higher_order} for $M=2,3$ and various grid sizes~$N$.}\label{tb:BM_results}
\end{center}
 \end{table}

We next consider once again the Dirichet and Neumann problems posed in the exterior of the unit sphere, but now utilizing BEM discretizations of the associated BW~\eqref{BW} and BM~\eqref{BM} integral equations. The numerical results are summarized in Table~\ref{tb:BEM_results}. The close-form analytical density interpolation procedure of Section~\ref{sec:BEM} is used in all the examples included in this table. The discrete variational formulations corresponding to the BW and BM integral equations are given in~\eqref{eq:DVF} and~\eqref{eq:VF_BM}, respectively. We recall that the latter is here discretized as indicated in Section~\ref{sec:BEM}---by expressing it in terms of kernel-regularized single- and double-layer operators. As expected, these results demonstrate  that far-field errors exhibit second-order convergence rates for both boundary integral equations and interpolation order $M=0$~and~1, as the mesh size $h=\max_{T\in\mathcal T_h, i,j=1,2,3}|p_{\tau_i}-p_{\tau_j}|$ decreases. In fact, the errors obtained using the interpolation orders $M=0$~and~$1$ are almost identical. The closeness of the errors observed might be explained by the possible dominance of the Galerkin-BEM $\mathcal O(h^2)$ errors over the errors introduced by the numerical integration procedure.  The near-field errors, on the other hand, exhibit nearly second-order convergence rates for both orders $M=0$~and~$1$, with significantly smaller errors obtained for $M=1$. 

In order to demonstrate the accuracy of the Nystr\"om method when dealing with more complex geometries, we consider  the scattering of a planewave $u^\inc(\ner) = \e^{\im\! k\ner\cdot \bol d}$, in the direction $\bol d=(\cos\frac\pi3,-\sin\frac\pi3,0)$, that impinges on the three sound-soft obstacles shown in Figure~\ref{fig:other}. The resulting scattered field is solution of the exterior Dirichlet  problem~\eqref{eq:Dirichlet} that is here solved by means of the Nystr\"om method applied to the kernel-regularized BW integral equation~\eqref{eq:BW_hoss}. The far-field errors reported in Figure~\ref{fig:other}  were produced by~\eqref{eq:error_formula} with the set $\Xi$ being the spherical grid displayed in Figure~\ref{fig:exp3a}. The reference solution $u^s_{\rm exact}$ in~\eqref{eq:error_formula} was generated using a fine discretization of the BW integral equation consisting of Chebyshev grids comprising $36\times 36$ points per surface patch. Both (algebraic) numerical (with $M=2$) and (closed-form) analytical (with $M=1$) density interpolation procedures are utilized in this example.  Third- and fifth-order convergence rates of the far-field errors are observed for $M=1$~and~$2$, respectively, for the smooth surface cases (bean and ellipsoid). Only second-order convergence is achieved in the non-smooth surface case (cube) for both interpolation orders $M=1$~and~$2$. The poor convergence rate observed in the latter case is explained by the singular behavior of the integral equation solution $\varphi$ along the edges of the cube~(cf.~\cite{Costabel:2000bd}). In fact, for the interior point source problem described above in this section---in which case $\varphi$ is smooth up to the edges of the cube---third- and fifth-order convergence rates are attained. The total field solution of a higher-frequency scattering problem for the bean-shaped obstacle---whose diameter is $10\lambda$ ($k=2\pi/\lambda =10\pi$)---is shown in Figure~\ref{fig:HF}. The near fields displayed in Figure~\ref{fig:HF} are accurate to at least four decimal places, everywhere including near and on the surface of the bean obstacle.

To finalize this section, we present examples aiming at demonstrating the capability of the BEM solver of handling complex geometries of engineering relevance. To this end, we consider a triangular mesh representation of a Falcon airplane produced by Gmsh~\cite{geuzaine2009gmsh}, which is used in the solution of two Dirichlet problems~\eqref{eq:Dirichlet} with different incident fields. In the first example we validate our BEM solver for this challenging geometry by considering an incident field given by two point sources placed inside the airplane's fuselage. The numerical solution is then compared with the exact solution. For the wavenumber $k=0.5\pi$ and the mesh size $h=1.62$ we obtain a relative error~\eqref{eq:error_formula}  of $8.1\cdot 10^{-2}$ at a sphere containing the airplane. The near-field error at two different planes intersecting the airplane is displayed on the first row of Figure~\ref{fig:falcon}. Finally, our second example considers a planewave incident field in the direction $\bol d = (\cos\frac\pi4,-\sin\frac\pi4,0)$ for the same wavenumber ($k=0.5\pi$).  Two views of the real part of total field are displayed in the second row of Figure~\ref{fig:falcon} where it can be clearly seen the acoustic shadow under the airplane.
  \begin{table}
   \begin{center}
     \scalebox{0.95}{\begin{tabular}{c|c|c|c|c|c}
\multicolumn{6}{c} {Boundary element method } \\     
\toprule
&& \multicolumn{2}{c|}{Far-field}&\multicolumn{2}{c}{Near-field} \\\cline{3-6}
$h$&DoF &$M=0$&$M=1$ &$M=0$&$M=1$\\\cline{3-6}
&&\multicolumn{4}{c} {BW integral equation} \\ \hline     
$5.34\cdot 10^{-1}$& 114 &$4.48\cdot 10^{-2}$&$4.50\cdot 10^{-2}$ &$7.11\cdot 10^{-2}$&$5.98\cdot 10^{-2}$\\
$2.72\cdot 10^{-1}$&400 &$1.29\cdot 10^{-2}$&$1.32\cdot 10^{-2}$&$2.70\cdot 10^{-2}$&$1.65\cdot 10^{-2}$\\
$1.36\cdot 10^{-1}$&1507 &$3.54\cdot 10^{-3}$&$3.54\cdot 10^{-3}$&$2.76\cdot 10^{-2}$&$4.14\cdot 10^{-3}$\\
$6.87\cdot 10^{-2}$&6009 &$8.98\cdot 10^{-4}$&$8.96\cdot 10^{-4}$&$4.81\cdot 10^{-3}$&$1.05\cdot 10^{-3}$\\
\hline
&&\multicolumn{4}{c} {BM integral equation} \\ \hline     
$5.34\cdot 10^{-1}$& 114 &$ 5.55\cdot 10^{-2}$&$5.38\cdot 10^{-2}$ &$8.80\cdot 10^{-2}$&$8.29\cdot 10^{-2}$\\
$2.72\cdot 10^{-1}$&400 &$1.65\cdot 10^{-2}$&$1.61\cdot 10^{-2}$&$3.02\cdot 10^{-2}$&$1.94\cdot 10^{-2}$\\
$1.36\cdot 10^{-1}$&1507 &$3.97\cdot 10^{-3}$&$3.95\cdot 10^{-3}$&$3.04\cdot 10^{-2}$&$5.49\cdot 10^{-3}$\\
$6.87\cdot 10^{-2}$&6009 &$1.02\cdot 10^{-3}$&$1.01\cdot 10^{-3}$&$4.71\cdot 10^{-3}$&$1.38\cdot 10^{-3}$\\
 \bottomrule
\end{tabular}}
\caption{\label{tb:convergence_near_far_BEM}  Far- and near-field relative errors in the solution of the Dirichlet problem~\eqref{eq:Dirichlet}  produced by the boundary element discretization of the BW integral equation~\eqref{BW} using the analytical PWDI technique for different interpolation orders $M$ and mesh sizes~$h$. The surface $\Gamma$ considered in this example is a sphere of unit radius and centered at the origin.}\label{tb:BEM_results}
\end{center}
 \end{table}

 \begin{figure}[h!]
\centering	
\includegraphics[scale=0.65]{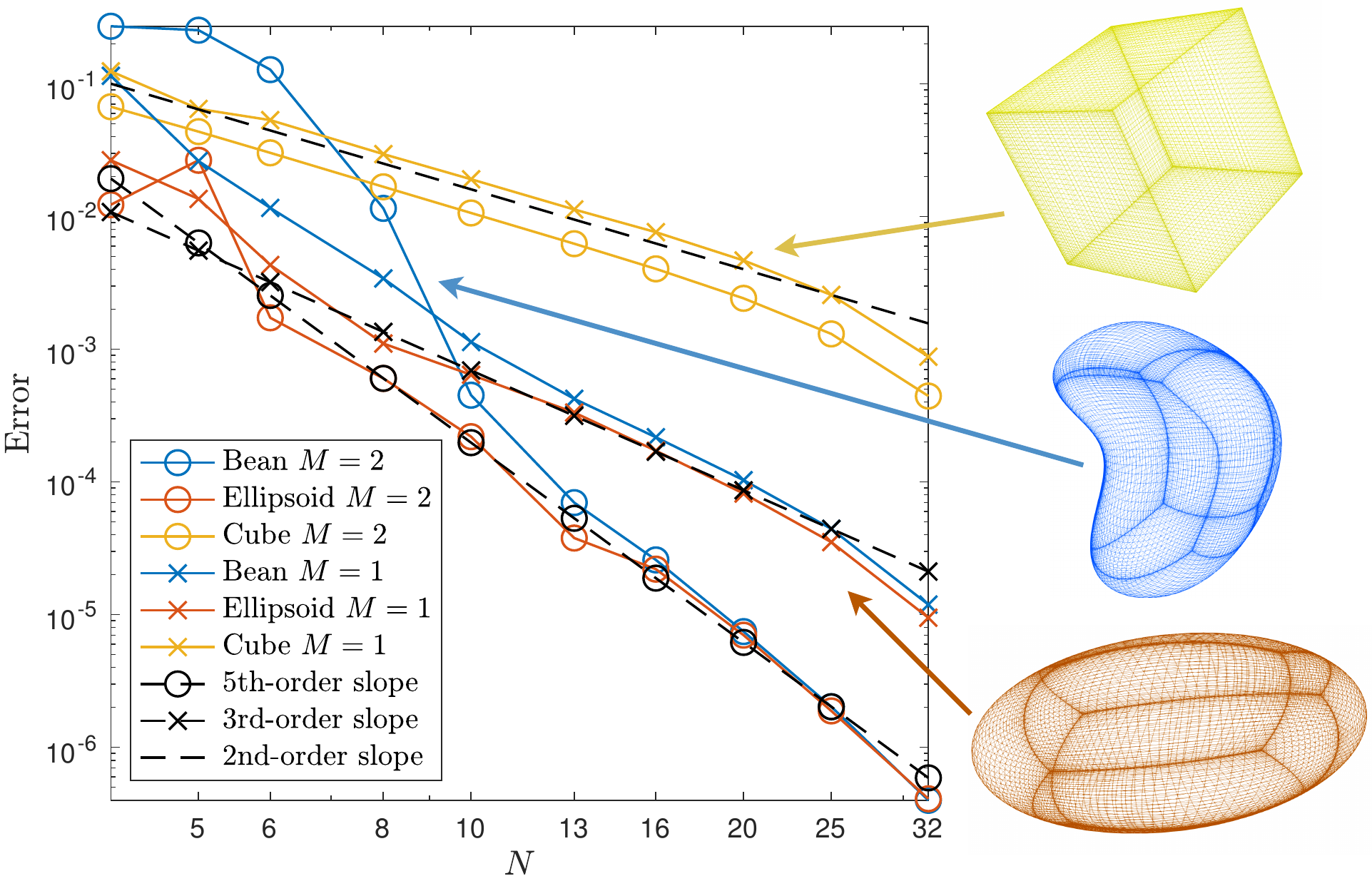}
\caption{Far-field errors in the solution of the Dirichlet problem~\eqref{eq:Dirichlet} corresponding to the scattering of a plane-wave off three different surfaces using the Chebyshev-based Nystr\"om method of Section~\ref{sec:Nystrom} applied to the BW integral equation~\eqref{BW} with  $k=\eta=1$. Both analytical (with $M=1$) and numerical  (with $M=2$) PWDI procedures were used in this example.}\label{fig:other}
\end{figure}

 \begin{figure}[h!]
\centering	
\includegraphics[scale=1]{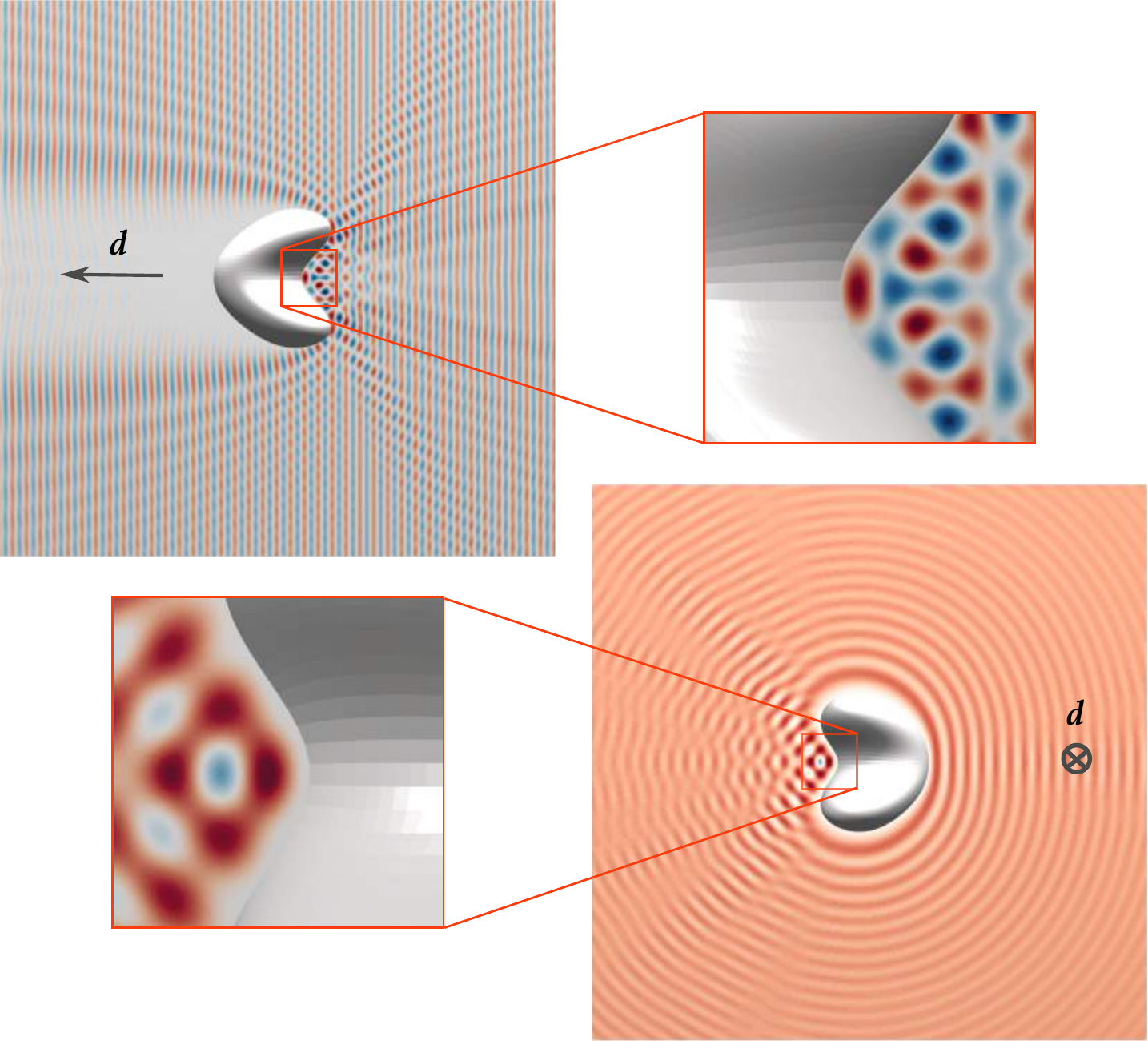}
\caption{Real part of the total field $u=u^s_D+u^\inc$ where $u^s_D$ is solution of the exterior Dirichlet problem~\eqref{eq:Dirichlet} corresponding to the scattering of the planewave $u^\inc(\ner)=\e^{\im\!k\ner\cdot \bol d}$, with $k=10\pi$ and $\bol d=(\cos\pi/3,-\sin\pi/3,0)$, off of a bean-shaped obstacle. Top: total field at a plane parallel to $\bol d$ that passes through the center of the obstacle. Bottom: total field at a plane perpendicular $\bol d$ that passes through the center of the obstacle. The surface $\Gamma$ was discretized using six $48\times 48$ Chebyshev grids. The discretized BW integral equation (with $\eta=k$) was solved by means of GMRES, which required 55 iterations to attain the prescribed $10^{-7}$ error tolerance.}\label{fig:HF}
\end{figure}
\subsection{Composite surfaces} In this final section we apply the multiple-scattering approach put forth in Section~\ref{sec:composites}, to the solution of the Dirichlet problem~\eqref{eq:Dirichlet}, posed in the exterior of the composite domain $\Omega = \Omega_1\cup\Omega_2$ with boundary $\Gamma=\p(\Omega_1\cup\Omega_2)$, where $\Omega_1$ is a sphere or radius~$0.5$ and $\Omega_2$ is a hemisphere of radius~$1.5$. The incident field used in this example is a planewave $u^{\inc}(\ner) = \e^{\im k\bol d\cdot \ner}$ in the direction $\bol d = (\cos\frac\pi4,0,-\sin\frac\pi4)$ and $k=\eta=1$. The multiple-scattering BW integral equation~\eqref{eq:IE_sep}, posed on $\tilde\Gamma = \p\Omega_1\cup\p\Omega_2$, is discretized using the BEM detailed in Section~\ref{sec:BEM} with $M=1$. Figure~\ref{fig:BEM_final} presents the far-field errors for various mesh sizes~$h$. The error is defined here as in~\eqref{eq:error_formula} with $\Xi$ being the spherical grid shown in Figure~\ref{fig:exp3a} and the reference solution $u^s_{\rm exact}$ being produced using a fine mesh discretization, with $h=0.11$, of the surfaces $\Gamma_1$ and $\Gamma_2$.  Three different $\Omega$ configurations, shown in inset plots in Figure~\ref{fig:BEM_final}, are considered, including one (on the left-hand-side) in which the two obstacles are touching at a single point. Clearly, second-order convergence is achieved, as $h$ decreases, in all three configurations. The real part of the total field solution of the problem of scattering (for $k=\eta=5$) together with the absolute value of the error obtained using the multiple-scattering approach and the standard approach, are shown in Figure~\ref{fig:BEM_final_final}. The reference solution for the error estimation is produced using a fine-grid discretization of $\Gamma$ using $h=0.1$.

 \begin{figure}[h!]
\centering	
\includegraphics[scale=0.26]{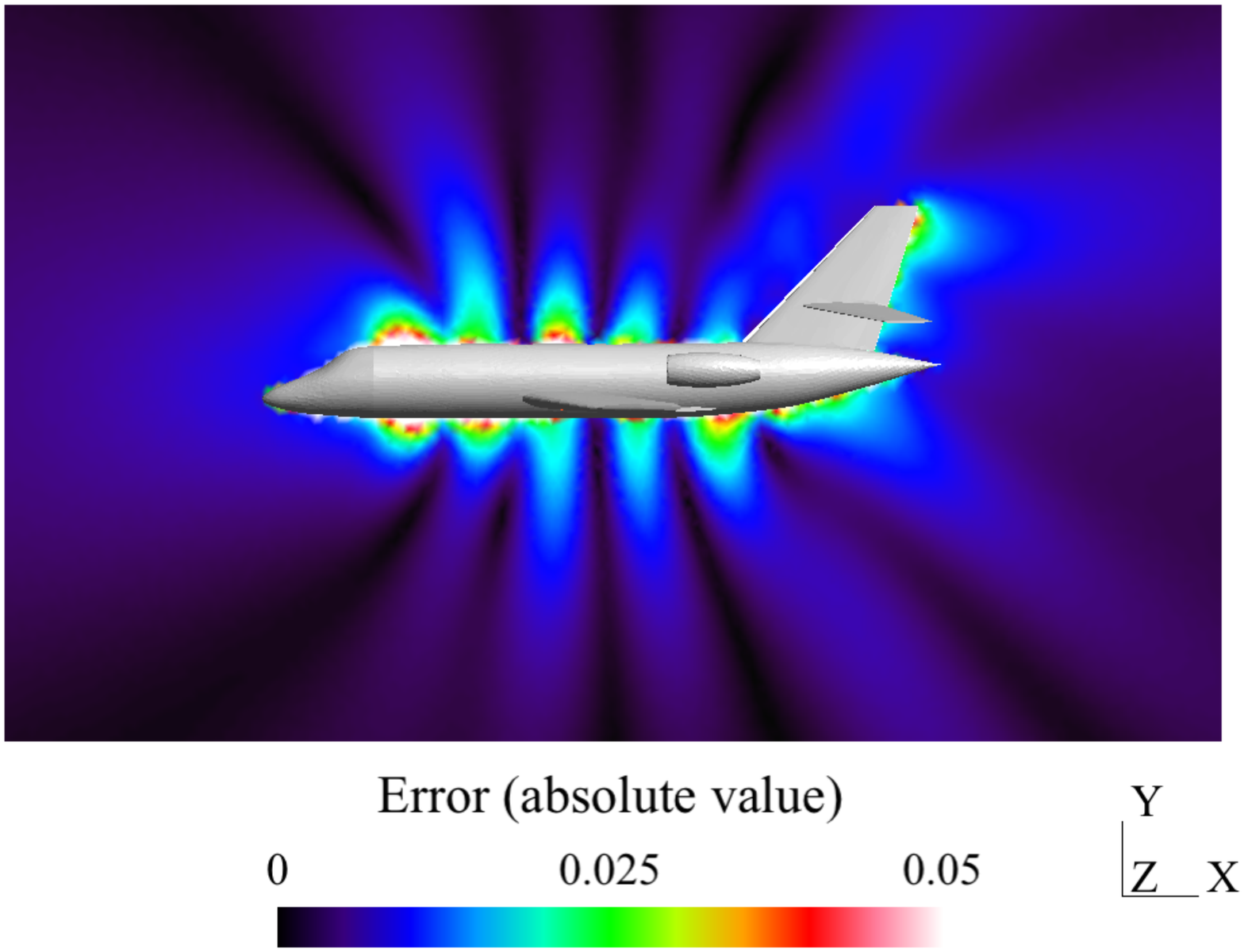}
\includegraphics[scale=0.26]{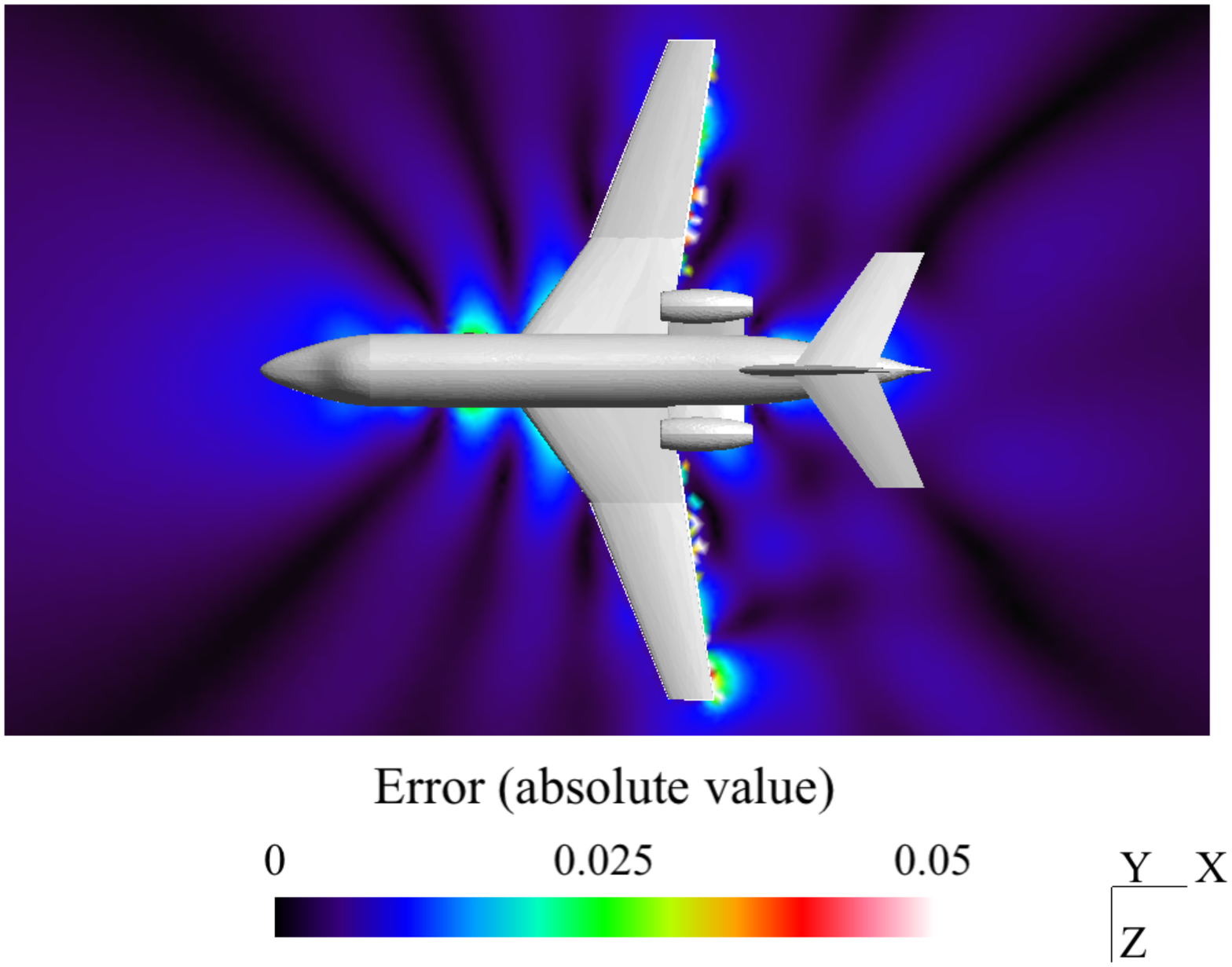}\\
\includegraphics[scale=0.26]{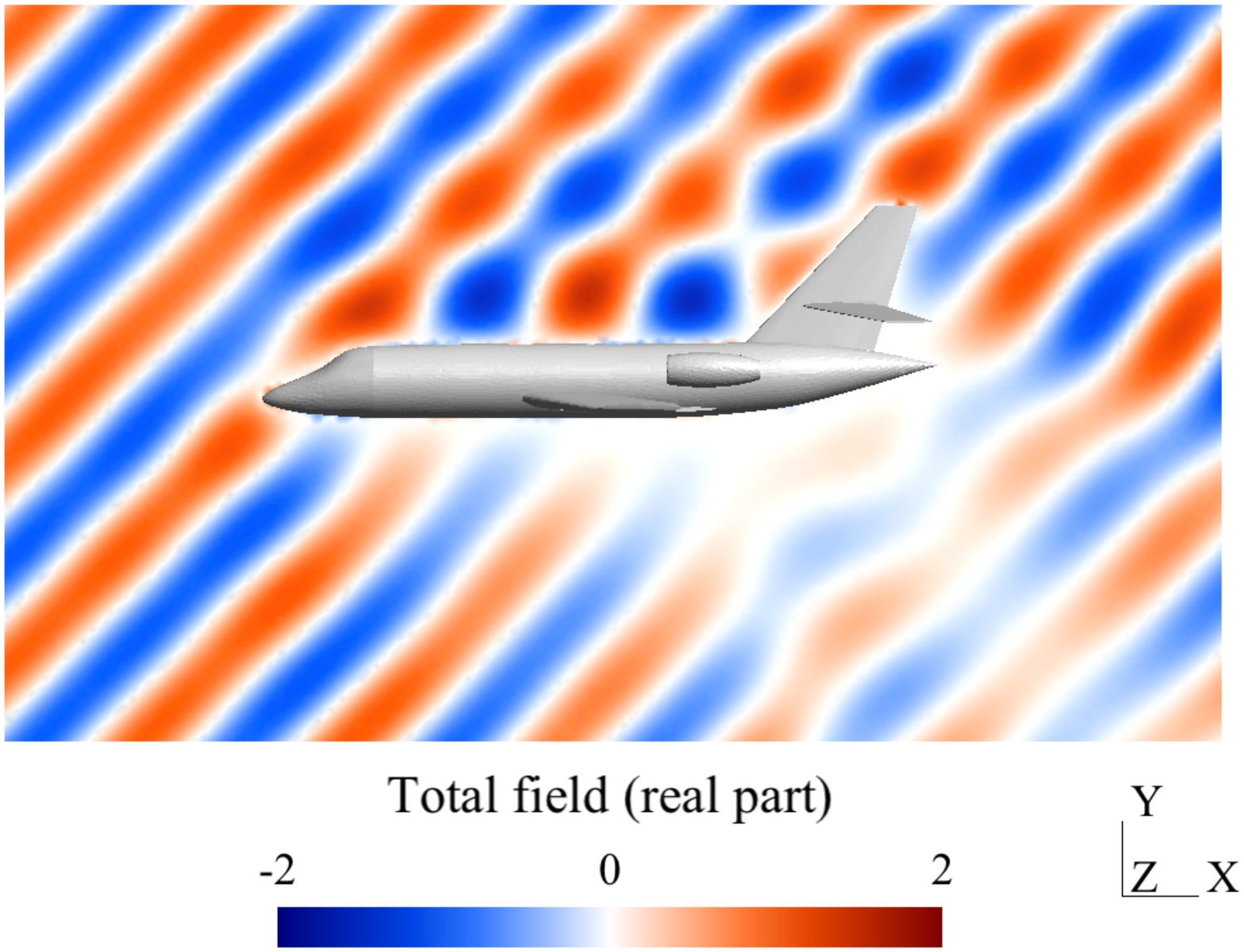}
\includegraphics[scale=0.26]{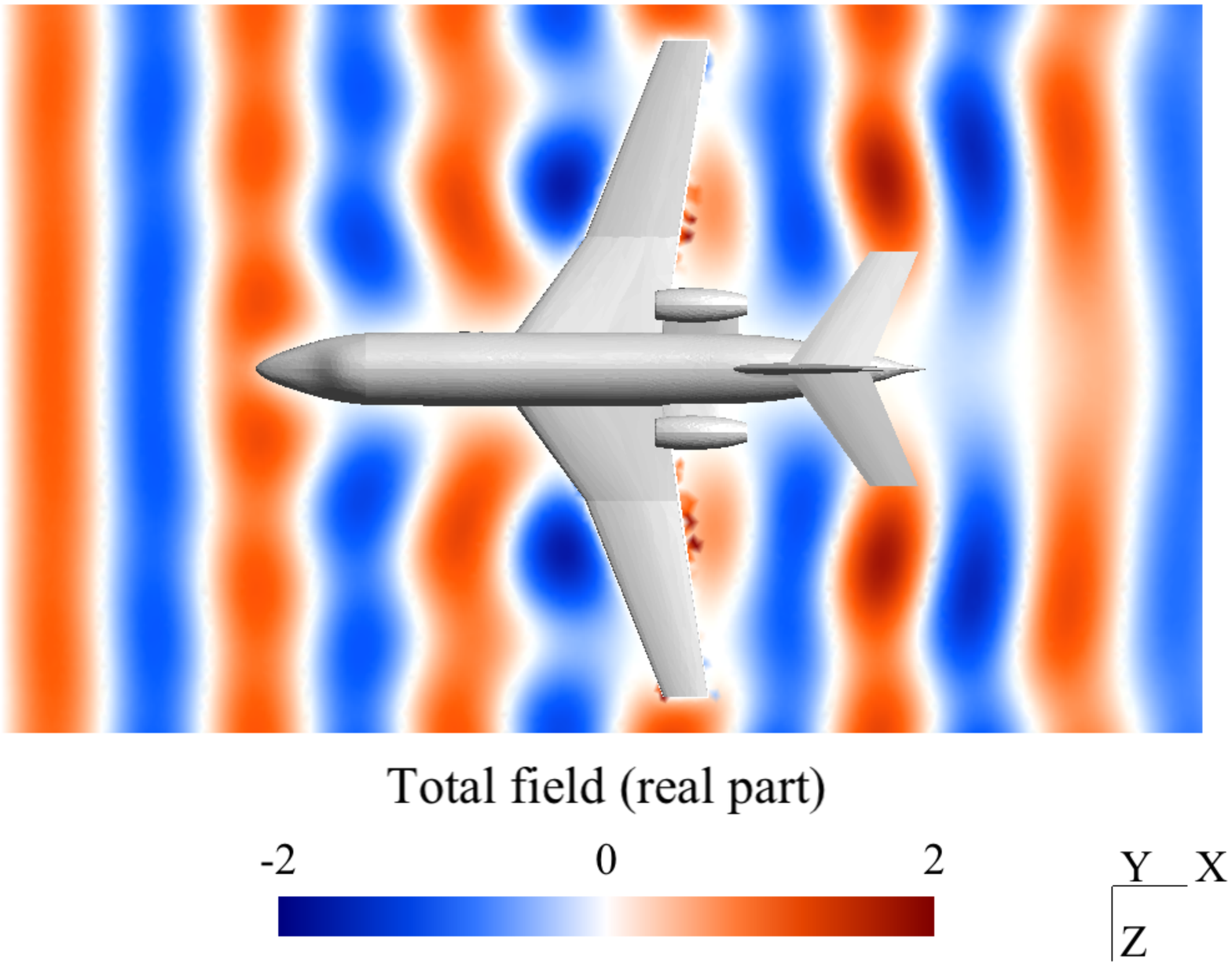}
\caption{First row: Two views of the absolute error in the BEM solution of~\eqref{eq:Dirichlet} with $k=0.5\pi$ for an incident field corresponding to two point sources placed inside the surface~$\Gamma$ which models a Falcon airplane. The relative error~\eqref{eq:error_formula} on a sphere containing the airplane is~$8.1\cdot10^{-2}$ in this example, where the mesh size is $h=1.62$. Second row: Two views of the real part of the total field solution of the problem of scattering~\eqref{eq:Dirichlet} for a planewave incident field in the direction $\bol d = (\cos\frac\pi4,-\sin\frac\pi4,0)$.  The analytical PWDI procedure with $M=1$ was used in all these examples.}\label{fig:falcon}
\end{figure}

 \begin{figure}[h!]
\centering	
\includegraphics[scale=0.7]{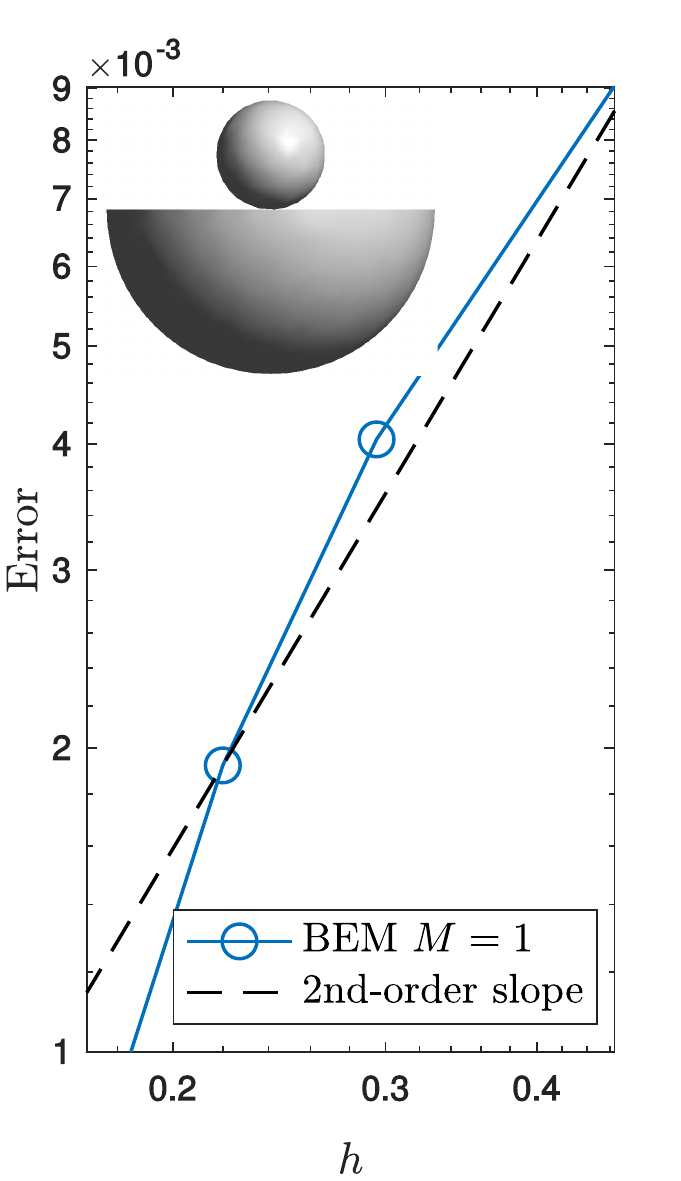}
\includegraphics[scale=0.7]{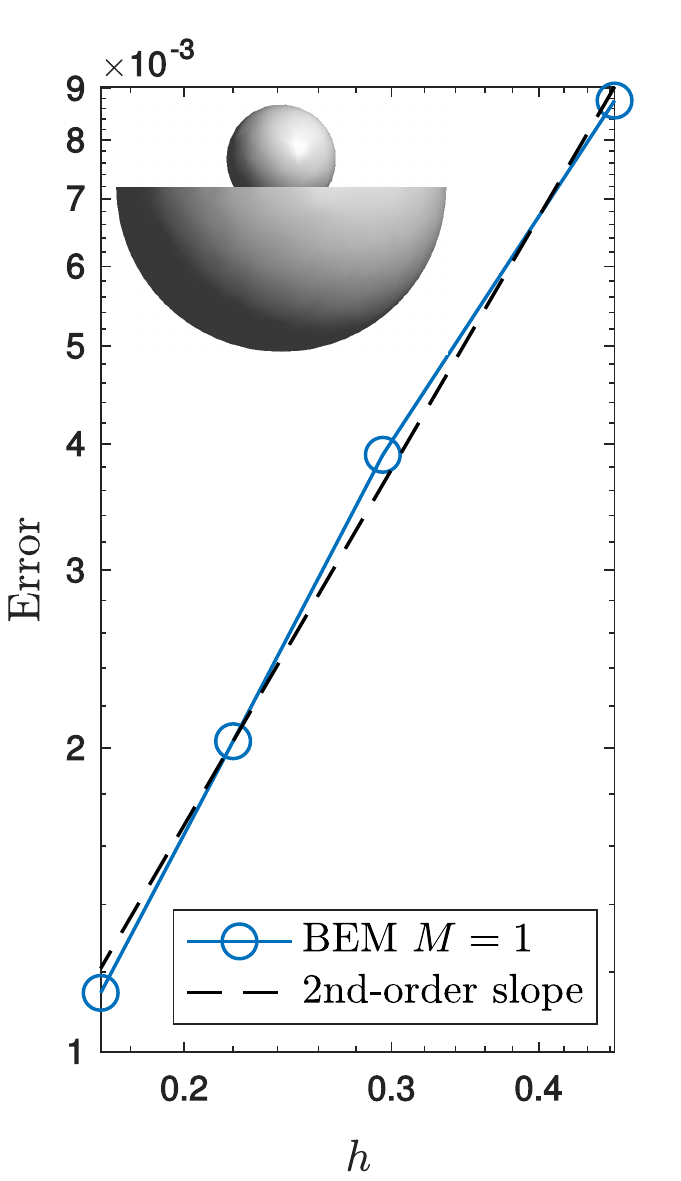}
\includegraphics[scale=0.7]{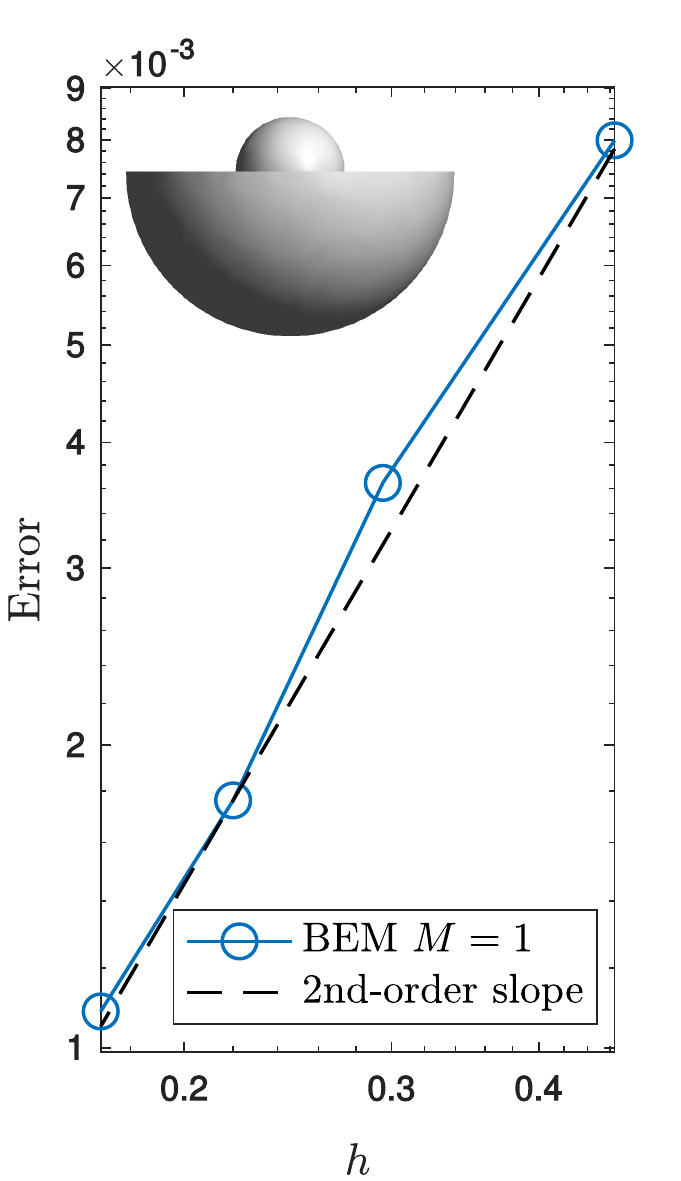}
\caption{Far-field errors in the solution of the Dirichlet problem~\eqref{eq:Dirichlet} corresponding to the scattering of a plane-wave in the direction $\bol d = (\cos\frac\pi4,0,-\sin\frac\pi4)$, off three different composite structures using the BEM method of Section~\ref{sec:Nystrom} applied to the BW integral equation~\eqref{BW} with  $k=\eta=1$.  The analytical PWDI procedure with $M=1$ was used in this example. Separate meshes of the sphere and the hemisphere where used in this example.}\label{fig:BEM_final}
\end{figure}

 \begin{figure}[h!]
\centering	
\includegraphics[scale=0.24]{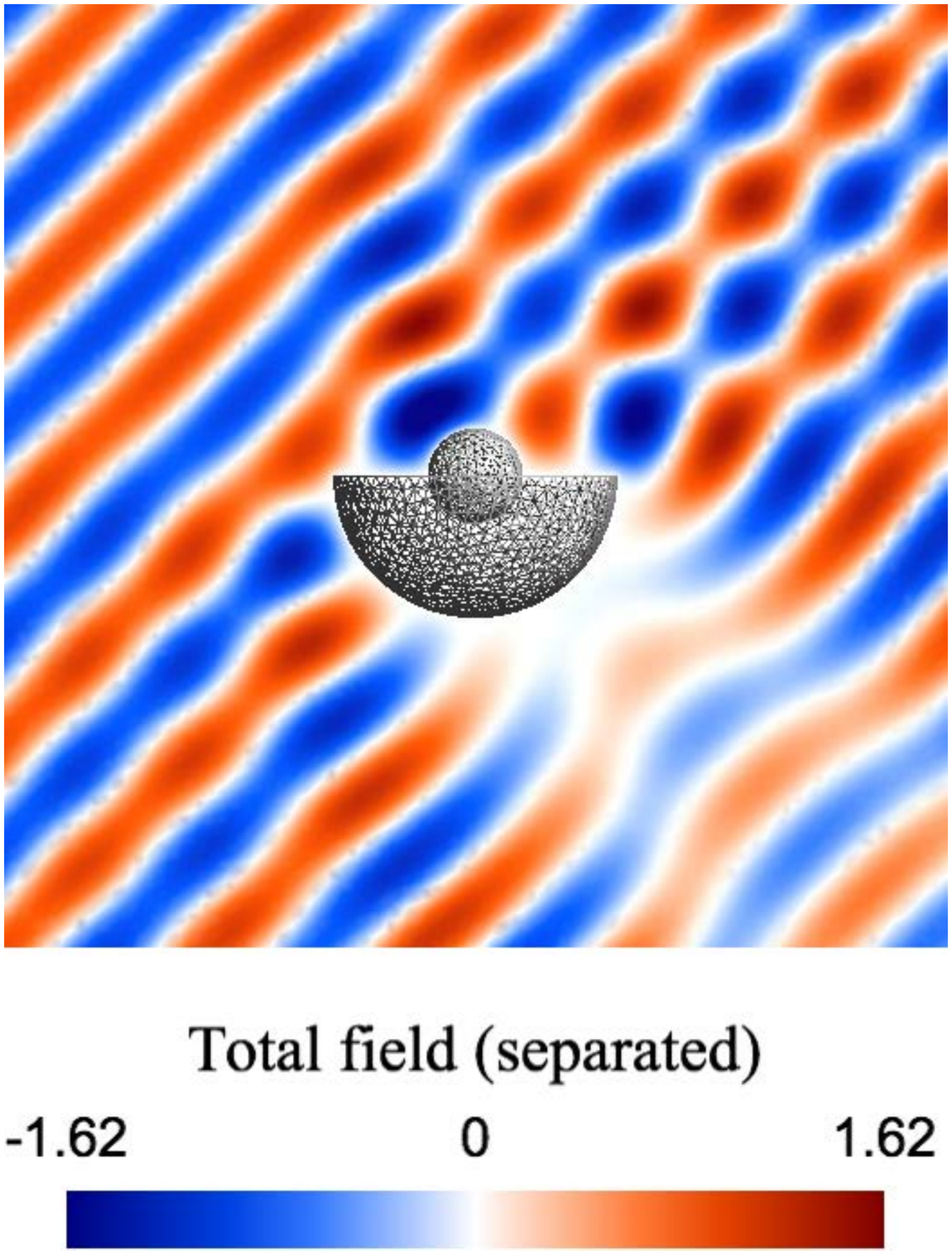}
\includegraphics[scale=0.24]{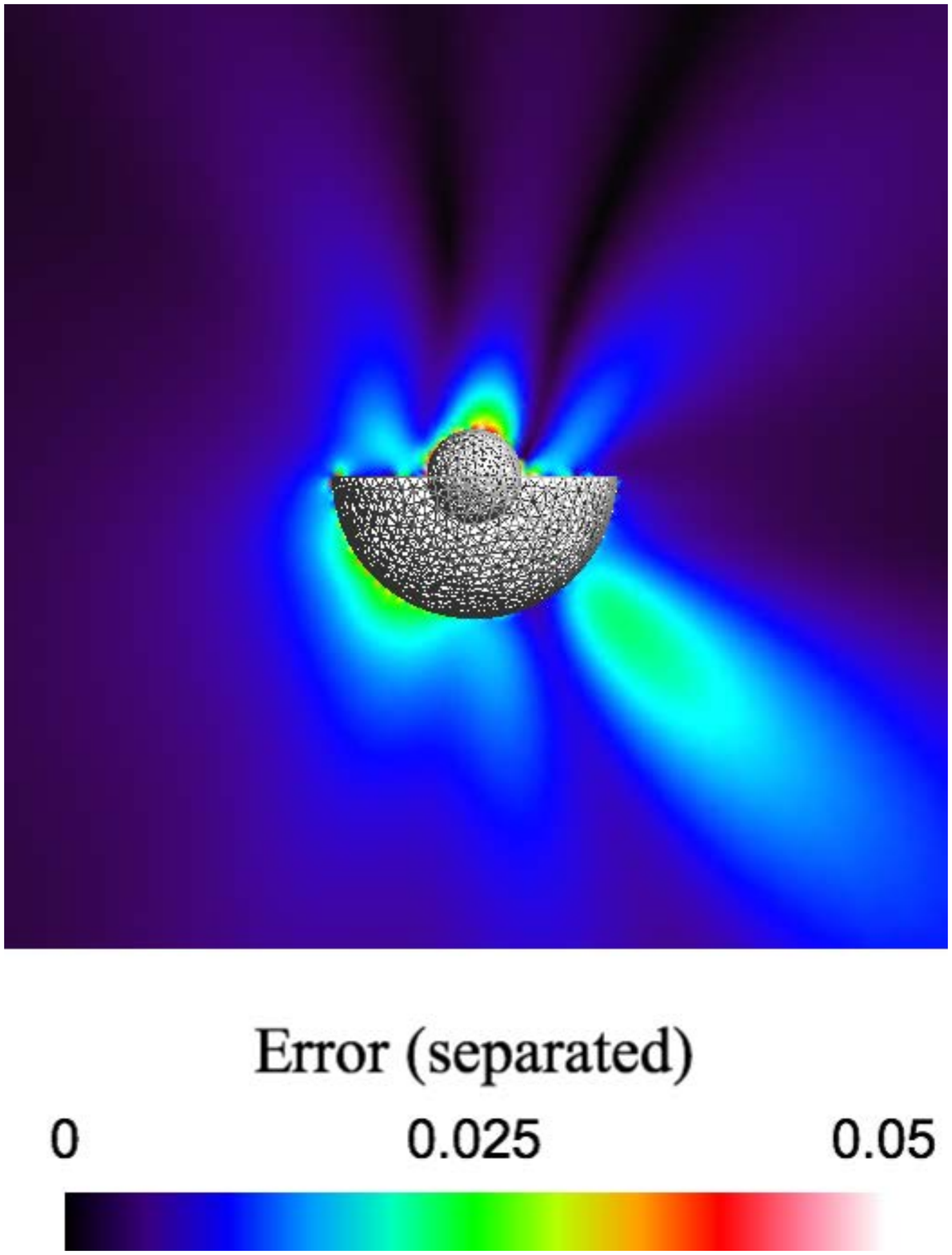}
\includegraphics[scale=0.24]{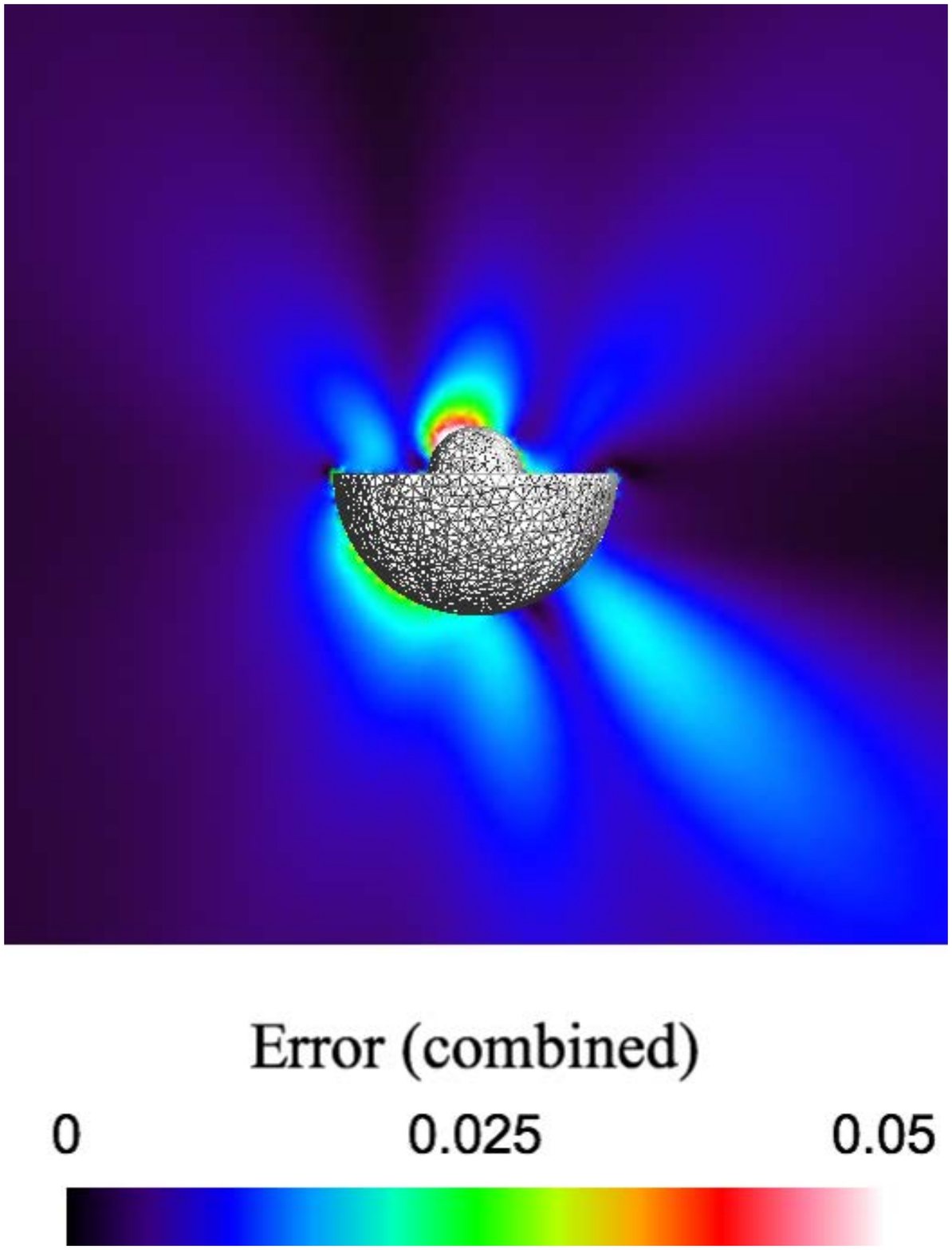}
\caption{Left: total near field  corresponding to the scattering of a plane-wave in the direction $\bol d = (\cos\frac\pi4,0,-\sin\frac\pi4)$ using the BEM method of Section~\ref{sec:Nystrom} applied to the BW integral equation~\eqref{BW} with  $k=\eta=4$ using separated meshes for the upper sphere and lower hemisphere. Center: error in the solution using separated meshes with $h=0.22$. Right: error in the solution using a  combined mesh with $h=0.21$. Reference solution computed using a combined mesh with $h= 0.1$. The analytical PWDI procedure with $M=1$ was used in this example.}\label{fig:BEM_final_final}
\end{figure}

\appendix
\section{Uniqueness of multiple scattering integral equation formulation for composite surfaces}\label{app:uniqueness}
This appendix is devoted to establishing uniqueness of solution of the boundary integral equation system~\eqref{eq:IE_sep}.  We present the following result whose proof asumes that $\Omega_1\cap\Omega_2\neq\emptyset$ is open:
\begin{lemma} The integral equation system~\eqref{eq:IE_sep} has at most one solution.
\end{lemma}
\begin{proof}
Assume that \eqref{eq:IE_sep} admits a non-trivial solution $\tilde\varphi$ when $(f_1,f_2)=(0,0)$. The potential $u^s(\ner) = (\tilde{\mathcal D}\tilde\varphi)(\ner) - \im\eta(\tilde{\mathcal S}\tilde\varphi)(\ner)$ is then a homogeneous solution of the Helmholtz equation in $\R^3\setminus(\Gamma_1\cup\Gamma_2)$ whose limits from outside vanish on both $\Gamma_1$ and $\Gamma_2$. Since the uniqueness of the exterior Dirichlet problem implies that $u^s=0$ in $\R^3\setminus(\Omega_1\cup\Omega_2)$~(cf.~\cite{COLTON:1983}), we have from the integral equation and the jump conditions of the single- and double-layer potentials on~$\Gamma_j$,  that $[u^s]=\varphi_j$ and $[\p_n u^s] = \im\!\eta\varphi_j$ on $\Gamma_j$, $j=1,2$. Therefore, $u^s$ is a homogeneous solution of the Helmholtz equation in  $\Omega_1\cap\Omega_2$ that satisfies the Robin boundary condition $\p_n u^s-\im\!\eta u^s=0$ on $\p(\Omega_1\cap\Omega_2)\subset\Gamma_1\cup\Gamma_2$. By uniqueness of this interior boundary value problem~(cf.~\cite{COLTON:1983}) we have that $u^s=0$ in ${\Omega_1\cap\Omega_2}$.  Similarly, we get that $u^s=0$ in ${\Omega_1\setminus\Omega_2}$ and in ${\Omega_2\setminus\Omega_1}$, and thus $u^s=0$ in all of ${\Omega_1\cup\Omega_2}$. Therefore, since $u^s=0$ in $\R^3\setminus (\Gamma_1\cup\Gamma_2)$ we conclude that $[u^s]=\varphi=0$ on $\Gamma_1\cup\Gamma_2$, which is a contradiction. The proof is now complete.
\end{proof}

\section{Explicit expression of the matrix $\bol C(p)$ for $M=3$}\label{app:explicit_M_3} This appendix is devoted to the derivation of the matrix $\bol C$ used in the algebraic procedure introduced in Section~\ref{sec:higher_order} for the construction of planewave interpolant in the case $M=3$. 

In order to ease derivations we define
$$w_\ell(q,p) = \exp\lf({{\rm i}k\bol d_\ell\cdot(q-p)}\rg) \andtext w_{n,\ell}(q,p) ={\rm i}k\lf(\bol d_\ell\cdot \bold n(q)\rg) w_\ell(q,p),$$ where $\bol d_\ell$ ($|\bol d_\ell|=1$) is a given planewave direction that does not depend on the expansions point $p$.
Evaluating all the partial derivatives of $w_\ell$ and $w_{n,\ell}$  up to third order, we get
\begin{subequations}\begin{equation}\begin{array}{rlll}
\p^\alpha w_\ell(q,p) &=& {\rm i}k \lf(\bol d_\ell\cdot \p^\alpha\bnex(q)\rg)w_\ell(q,p),\smallskip\\
\p^{\beta+\alpha} w_\ell (q,p) &=& {\rm i}k \lf(\bol d_\ell\cdot \p^{\alpha+\beta}\bnex(q)\rg) w_\ell(q,p)+{\rm i}k \lf(\bol d_\ell\cdot \p^\alpha\bnex(q)\rg)\p^\beta w_\ell(q,p),\smallskip\\
\p^{\gamma+\beta+\alpha} w_\ell (q,p) &=& {\rm i}k \lf(\bol d_\ell\cdot \p^{\gamma+\beta+\alpha}\bnex(q)\rg) w_\ell(q,p)+{\rm i}k \lf(\bol d_\ell\cdot \p^{\beta+\alpha}\bnex(q)\rg)\p^\gamma w_\ell(q,p)\\
&&\!\!+{\rm i}k \lf(\bol d_\ell\cdot \p^{\gamma+\alpha}\bnex(q)\rg)\p^\beta w_\ell(q,p) +{\rm i}k \lf(\bol d_\ell\cdot \p^\alpha\bnex(q)\rg)\p^{\gamma+\beta} w_\ell(q,p),
\end{array}\end{equation}
and
\begin{equation}\begin{array}{rlll}
\p^\alpha w_{n,\ell} (q,p) &=&{\rm i}k\{ (\bol d_\ell\cdot \p^\alpha\bnor(q)) w_\ell(q,p) +\lf(\bol d_\ell\cdot \bnor(q)\rg)\p^\alpha w_\ell(q,p)\}, \smallskip\\
\p^{\beta+\alpha} w_{n,\ell} (q,p) &=&{\rm i}k\{(\bol d_\ell\cdot \p^{\beta+\alpha}\bnor(q)) w_\ell(q,p)+ (\bol d_\ell\cdot\p^\alpha \bnor(q))\p^\beta w_\ell(q,p)\smallskip \\
&&+(\bol d_\ell\cdot \p^\beta\bnor(q))\p^\alpha w_\ell(q,p)+(\bol d_\ell\cdot \bnor(q))\p^{\beta+\alpha} w_\ell(q,p)\}\smallskip\\
\p^{\gamma+\beta+\alpha} w_{n,\ell} (q,p) &=&{\rm i}k\{(\bol d_\ell\cdot \p^{\gamma+\beta+\alpha}\bnor(q)) w_\ell(q,p)+(\bol d_\ell\cdot \p^{\beta+\alpha}\bnor(q))\p^\gamma w_\ell(q,p)\smallskip\\
&&+ (\bol d_\ell\cdot\p^{\gamma+\alpha} \bnor(q))\p^\beta w_\ell(q,p)+(\bol d_\ell\cdot\p^\alpha \bnor(q))\p^{\gamma+\beta} w_\ell(q,p)\smallskip \\
&&+(\bol d_\ell\cdot \p^{\gamma+\beta}\bnor(q))\p^\alpha w_\ell(q,p)+(\bol d_\ell\cdot \p^\beta\bnor(q))\p^{\gamma+\alpha} w_\ell(q,p)\smallskip\\
&&+(\bol d_\ell\cdot \p^{\gamma}\bnor(q))\p^{\beta+\alpha} w_\ell(q,p)+(\bol d_\ell\cdot \bnor(q))\p^{\gamma+\beta+\alpha} w_\ell(q,p)\}\smallskip\\
\end{array}\end{equation}\label{eq:der_plane_wave}\end{subequations}
 where the indices $\alpha,\beta,\gamma\in\mathbb Z_+^2$ satisfy   $|\alpha|=|\beta|=|\gamma|=1$.

Selecting the ordering 
\begin{equation}\begin{split}\alpha:&~ (0,0)\mapsto 1,\quad (1,0)\mapsto 2,\quad (0,1)\mapsto 3,\quad (2,0)\mapsto4,
\quad(1,1)\mapsto 5,\\
&~(0,2)\mapsto 6,\quad(3,0)\mapsto 7,\quad (2,1)\mapsto 8,\quad (1,2)\mapsto 9,\quad (0,3)\mapsto 10,\end{split}\end{equation}
for the multi-indices, and enforcing the point conditions by evaluating the expressions~\eqref{eq:der_plane_wave} at $q=p$ we obtain that the entries of  $\bol C(p)$, that are denoted by  $c_{n,\ell}$ for $n=1,\ldots,N$ and $\ell=1,\ldots,L$, are given by the following (recursive) relations
\begin{subequations}\begin{equation*}\begin{array}{rlll}
c_{1,\ell} &:=&\p^{(0,0)} w_\ell (p,p)= 1,\smallskip\\
c_{2,\ell} &:=&\p^{(1,0)} w_\ell (p,p)= {\rm i}k \lf(\bol d_\ell\cdot \p^{(1,0)}\bnex(p)\rg),\smallskip\\
c_{3,\ell} &:=&\p^{(0,1)} w_\ell (p,p)= {\rm i}k \lf(\bol d_\ell\cdot \p^{(0,1)}\bnex(p)\rg),\smallskip\\
c_{4,\ell} &:=&\p^{(2,0)} w_\ell (p,p)= {\rm i}k \lf(\bol d_\ell\cdot \p^{(2,0)}\bnex(p)\rg)+ c_{2,\ell}^2,\smallskip\\
c_{5,\ell}&:=&\p^{(1,1)} w_\ell (p,p) = {\rm i}k \lf(\bol d_\ell\cdot \p^{(1,1)}\bnex(p)\rg)+c_{2,\ell}c_{3,\ell},\smallskip\\
c_{6,\ell}&:=&\p^{(0,2)} w_\ell (p,p) = {\rm i}k \lf(\bol d_\ell\cdot \p^{(0,2)}\bnex(p)\rg)+c_{3,\ell}^2,\smallskip\\
c_{7,\ell} &:=&\p^{(3,0)} w_\ell (p,p)= {\rm i}k \lf(\bol d_\ell\cdot \p^{(3,0)}\bnex(p)\rg)+3c_{2,\ell}c_{4,\ell}-2c_{2,\ell}^3,\smallskip\\
c_{8,\ell}&:=&\p^{(2,1)} w_\ell (p,p) = {\rm i}k \lf(\bol d_\ell\cdot \p^{(2,1)}\bnex(p)\rg)+2c_{2,\ell}(c_{5,\ell}-c_{2,\ell}c_{3,\ell}) +c_{3,\ell}c_{4,\ell},\smallskip\\
c_{9,\ell}&:=&\p^{(1,2)} w_\ell (p,p) = {\rm i}k \lf(\bol d_\ell\cdot \p^{(1,2)}\bnex(p)\rg)+2c_{3,\ell}(c_{5,\ell}-c_{2,\ell}c_{3,\ell})+c_{2,\ell}c_{6,\ell},\smallskip\\
c_{10,\ell}&:=&\p^{(0,3)} w_\ell (p,p) = {\rm i}k \lf(\bol d_\ell\cdot \p^{(0,3)}\bnex(p)\rg)+3c_{3,\ell}c_{6,\ell}-2c_{3,\ell}^3,\smallskip
\end{array}\end{equation*}
and 
\begin{equation*}\begin{array}{rlll}
c_{11,\ell}&:=&\p^{(0,0)} w_{n,\ell} (p,p)={\rm i}k\lf(\bol d_\ell\cdot \bold n(p)\rg) , \smallskip\\
c_{12,\ell}&:=&\p^{(1,0)} w_{n,\ell} (p,p) ={\rm i}k\lf(\bol d_\ell\cdot \p^{(1,0)}\bnor(p)\rg) +c_{2,\ell}c_{11,\ell}, \smallskip\\
c_{13,\ell}&:=&\p^{(0,1)} w_{n,\ell} (p,p) ={\rm i}k\lf(\bol d_\ell\cdot \p^{(0,1)}\bnor(p)\rg) +c_{3,\ell}c_{11,\ell}, \smallskip\\
c_{14,\ell}&:=&\p^{(2,0)} w_{n,\ell} (p,p) ={\rm i}k\lf(\bol d_\ell\cdot \p^{(2,0)}\bnor(p)\rg)+ 2c_{2,\ell}(c_{12,\ell}-c_{2,\ell}c_{11,\ell})+c_{4,\ell}c_{11,\ell},\smallskip\\
c_{15,\ell}&:=&\p^{(1,1)} w_{n,\ell} (p,p) ={\rm i}k\lf(\bol d_\ell\cdot \p^{(1,1)}\bnor(p)\rg)-2c_{2,\ell}c_{3,\ell}c_{11,\ell}+ c_{3,\ell}c_{12,\ell}+c_{2,\ell}c_{13,\ell}\\
&&+c_{5,\ell}c_{11,\ell},\smallskip\\
c_{16,\ell}&:=&\p^{(0,2)} w_{n,\ell} (p,p) ={\rm i}k\lf(\bol d_\ell\cdot \p^{(0,2)}\bnor(p)\rg)+2c_{3,\ell}(c_{13,\ell}-c_{3,\ell}c_{11,\ell})+c_{6,\ell}c_{11,\ell},\smallskip\\
\end{array}\end{equation*}\end{subequations}
\begin{subequations}\begin{equation*}\begin{array}{rlll}
c_{17,\ell}&:=&\p^{(3,0)} w_{n,\ell} (p,p) ={\rm i}k\lf(\bol d_\ell\cdot \p^{(3,0)}\bnor(p)\rg)+3(c_{2,\ell}c_{14,\ell}+c_{4,\ell}c_{12,\ell})\\
&&+6c_{2,\ell}(c_{2,\ell}^2c_{11,\ell}-c_{2,\ell}c_{12,\ell}-c_{4,\ell}c_{11,\ell})+c_{7,\ell}c_{11,\ell},\smallskip\\
c_{18,\ell}&:=&\p^{(2,1)} w_{n,\ell} (p,p) ={\rm i}k(\bol d_\ell\cdot \p^{(2,1)}\bnor(p))+6c^2_{2,\ell}c_{3,\ell}c_{11,\ell}-4c_{2,\ell}(c_{3,\ell}c_{12,\ell}\\
&&+c_{5,\ell}c_{11,\ell})+2(c_{2,\ell}(c_{15,\ell}-c_{2,\ell}c_{13,\ell})+c_{5,\ell}c_{12,\ell}-c_{3,\ell}c_{4,\ell}c_{11,\ell})\smallskip\\
&&+c_{3,\ell}c_{14,\ell}+c_{4,\ell}c_{13,\ell}+c_{8,\ell}c_{11,\ell},\smallskip\\
c_{19,\ell}&:=&\p^{(1,2)} w_{n,\ell} (p,p)={\rm i}k(\bol d_\ell\cdot \p^{(1,2)}\bnor(p))+6c_{2,\ell}c^2_{3,\ell}c_{11,\ell}-4c_{3,\ell}(c_{2,\ell}c_{13,\ell}\\
&&+c_{5,\ell}c_{11,\ell})+2(c_{3,\ell}(c_{15,\ell}-c_{3,\ell} c_{12,\ell})+c_{5,\ell}c_{13,\ell}-c_{2,\ell}c_{6,\ell}c_{11,\ell})\smallskip\\
&&+c_{2,\ell}c_{16,\ell}+c_{6,\ell}c_{12,\ell}+c_{9,\ell}c_{11,\ell},\smallskip\\
c_{20,\ell}&:=&\p^{(0,3)} w_{n,\ell} (p,p) ={\rm i}k(\bol d_\ell\cdot \p^{(0,3)}\bnor(p))+3(c_{3,\ell}c_{16,\ell}+c_{6,\ell}c_{13,\ell})\smallskip\\
&&+6c_{3,\ell}(c^2_{3,\ell}c_{11,\ell}-c_{3,\ell}c_{13,\ell}-c_{6,\ell}c_{11,\ell})+c_{10,\ell}c_{11,\ell}.\end{array}\end{equation*}\end{subequations}

\begin{remark}
Since the  identities $\p^{\alpha}\bnex = 0$ for $|\alpha|>1$ and $\p^{\alpha}\bnor = 0$ for $|\alpha|>0$ hold true for any planar coordinate patch, we have that  letting $\tau_1={\rm i}k \lf(\bol d_\ell\cdot \p^{(1,0)}\bnex\rg)$, $\tau_2={\rm i}k \lf(\bol d_\ell\cdot \p^{(0,1)}\bnex\rg)$ and $\tau_3={\rm i}k\lf(\bol d_\ell\cdot\bol n\rg)$, the $\ell$-th column of $\bol C$ corresponding to a given planewave direction $\bol d_\ell$ is simply given the vector $[\bold c,\tau_3\bold c]^T\in \C^{20}$ where
$\bold c =[1,\tau_1,\tau_2,\tau_1^2,\tau_1\tau_2,\tau_2^2,\tau_1^3,\tau_1^2\tau_2,\tau_1\tau_2^2,\tau_2^3]$.
\end{remark}

\section{Compatibility of the PWDI method with fast methods}\label{app:compatible}
This appendix discusses the compatibility of the PWDI technique with standard fast multipole methods (FMMs) for the acceleration of the solution of the combined field integral equation~\eqref{BW} by means of iterative linear algebra solvers such as~GMRES~\cite{saad1986gmres}.  In a nutshell, FMMs rely on a certain low-rank approximation of the relevant integral kernel $K:\Gamma\times\Gamma\to \C$ in the form
\begin{equation}\label{eq:low_rank}K(p,q) \approx \sum_{\ell=1}^P H_\ell(p)M_\ell(q),
\end{equation} which is typically valid for observation points $p$ that are sufficiently far away from the source point $q$. In the case of the BW combined field integral operator, for example, the kernel is given by
$K(p,q)=\partial G(p,q)/\partial n(q)-{\rm i}\eta G(p,q),$
where $G$ is the free-space Green function defined in~\eqref{eq:GF}.  When a kernel expansion of the form~\eqref{eq:low_rank} is available, a FMM can be devised to effectively reduce the computational cost of the numerical evaluation of  the integral operator
\begin{equation}
(T\varphi)(p) =\int_{\Gamma} K(p,q) \varphi(q)\de s(q),\qquad p\in\Gamma,\label{eq:integral_operator}
\end{equation}
from  $\mathcal O(N_t^2)$ to nearly $\mathcal O(N_t)$ operations, where $N_t$ denotes the total number of quadrature points  (resp. mesh nodes) used in the discretization of the boundary integral over $\Gamma$ by the Nystr\"om (resp. boundary element) method. At low-frequencies, when the kernel $K$ is  non-oscillatory (i.e., for obstacles $\Omega$ of just a few wavelengths $\lambda=2\pi/k$ in diameter), such approximation can be achieved, for example, by means of multipole/planewave expansions~\cite{gumerov2005fast,coifman1993fast,Greengard1998}, polynomial interpolation~\cite{fong2009black}, or projection methods based on equivalent sources~\cite{Ying:2004cv}. 

Unfortunately, however, unlike the non-regularized combined field integral equation~\eqref{BW}, the integral equation~\eqref{eq:BW_hoss} resulting from application of the density interpolation technique requires evaluation of boundary integrals of the form
\begin{equation}\begin{split}
\int_{\Gamma}\frac{\partial G(p,q)}{\partial n(q)}\lf\{\varphi(q)-u(q,p)\rg\}\de s(q)
\mbox{ and }
\int_{\Gamma} G(p,q)\lf\{{\rm i}\eta\varphi(q)-u_n(q,p)\rg\}\de s(q)
\end{split}\label{eq:reg_int}
\end{equation}
 for $p\in\Gamma$, where not only $G$ depends on $p$ and $q$ but also the Dirichlet and Neumann traces $u$ and $u_n$ of the interpolating function~$U$ in~\eqref{eq:interpolants}. 
Therefore, in principle, any FMM for the numerical evaluation of~\eqref{eq:reg_int}  would require  low-rank approximations of two additional kernels, namely 
$$\frac{\partial G(p,q)}{\partial n(q)}u(q,p)\andtext G(p,q)u_n(q,p).$$


Fortunately, direct use of existing FMMs for the evaluation of~\eqref{eq:reg_int} is still possible when the algorithm introduced in Section~\ref{sec:higher_order} is used for the construction of the interpolating function $U$. Indeed, since the directions $\bol d_\ell$ of the planewaves 
$$W_\ell(\ner-p)=\exp({\rm i}k\bol d_\ell\cdot(\ner-q))=\overline{w_\ell(p)}W_\ell(\ner),\quad \ell=1,\ldots,L,$$ are independent of the both observation and source points, the resulting interpolating functions is trivially separable, i.e, 
$$
U(\ner,p) =\sum_{\ell=1}^{L}\phi_\ell(p)\overline{w_\ell(p)}W_\ell(\ner),
$$
where the expansion coefficients are given by
$$
\phi_\ell(p) := \sum_{|\alpha|=0}^{M}\frac{\p^\alpha\varphi(p)}{\alpha!}\lf\{a_{\ell,\alpha}(p)+{\rm i}\eta b_{\ell,\alpha}(p)\rg\},\quad \ell=1,\ldots,L.
$$
in terms of  $a_{\ell,\alpha}$ and $b_{\ell,\alpha}$ being given by solution of the linear systems~\eqref{eq:lin_sym1} and~\eqref{eq:lin_sym2}, respectively.
Therefore, upon integration over $\Gamma$ we obtain 
\begin{equation}\begin{split}
\int_\Gamma \frac{\p G(p,q)}{\p n(q)}u(q,p)\de s(q) =&~\sum_{\ell=1}^{L}\phi_\ell(p)\overline{w_\ell(p)}\int_\Gamma \frac{\p G(p,q)}{\p n(q)} w_\ell(q)\de s(q),\\
\int_\Gamma G(p,q)u_n(q,p)\de s(q) =&~\sum_{\ell=1}^{L}\phi_\ell(p)\overline{w_\ell(p)}\int_\Gamma G(p,q)(\bnor(q)\cdot\bol d_\ell) w_{\ell}(q)\de s(q),
\end{split}\end{equation}
for all $p\in\Gamma$. We thus conclude from here that the kernel regularized integral operator~\eqref{eq:integral_operator} can in fact be evaluated by direct application of FMM-accelerated double- and single-layer operators, to each one of the planewaves $w_\ell$, $\ell=1,\ldots,L$. Furthermore, since the coefficients $a_{\ell,\alpha}$ and $b_{\ell,\alpha}$ for $|\alpha|\leq M$ and $\ell=1,\ldots,L$, the multipole moments (or equivalent sources), and the local expansions coefficients associated to each one of the planewave $w_\ell$, are independent of the density function $\varphi$, they can be precomputed, stored, reused at each iteration of the linear algebra solver (GMRES) thus further reducing the computational cost associated to a forward map evaluation. It can be easily checked, finally,  that the overall computational cost of one forward map evaluation using the proposed FMM-accelerated algorithm is $\mathcal O(N_t\log N_t)$ for the high-order Nystr\"om method of Section~\ref{sec:Nystrom}, and $\mathcal O(N_t)$ for the low-order BEM of Section~\ref{sec:BEM}.

\bibliographystyle{abbrv}
\bibliography{references}
\end{document}